%% file: thick.tex
\documentclass[a4paper,10pt]{amsart}
\usepackage{amsmath}
\usepackage{amssymb}
\usepackage[all]{xy}
 \usepackage[usenames,dvipsnames]{pstricks}
 \usepackage{epsfig}
 \usepackage{pst-grad} 
 \usepackage{pst-plot}

\title[Thick subcategories of algebraic triangulated categories]{Thick subcategories of finite algebraic triangulated categories}
\author{Claudia K\"ohler}
\address{University of Paderborn/University of Bielefeld,
 Germany} 
\email{claudia.koehler@math.upb.de}

\newtheorem{Thm}{Theorem}[section]
\newtheorem{Prop}[Thm]{Proposition}
\newtheorem{Lemma}[Thm]{Lemma}

\theoremstyle{definition}
\newtheorem{Def}[Thm]{Definition}

\newtheorem{Ex}[Thm]{Example}

\theoremstyle{remark}
\newtheorem{Rem}[Thm]{Remark}

\newcommand{\dotcup}{\ensuremath{\mathaccent\cdot\cup}}

\newcommand{\stmod}{\operatorname{\underline{mod}}\nolimits}
\renewcommand{\mod}{\operatorname{mod}\nolimits}

\newcommand{\typ}{\operatorname{typ}\nolimits}
\newcommand{\Ext}{\operatorname{Ext}\nolimits}
\newcommand{\Hom}{\operatorname{Hom}\nolimits}
\newcommand{\End}{\operatorname{End}\nolimits}
\newcommand{\cox}{\operatorname{cox}\nolimits}
\newcommand{\id}{\operatorname{id}\nolimits}

\newcommand{\proj}{\operatorname{proj}\nolimits}
\newcommand{\Mod}{\operatorname{Mod}\nolimits}
\newcommand{\dimvec}{\operatorname{\underline{dim}}\nolimits}
\renewcommand{\dim}{\operatorname{dim}\nolimits}

\newcommand{\Cat}{\operatorname{Cat}\nolimits}
\newcommand{\NC}{\operatorname{NC}\nolimits}
\newcommand{\q}{\vec}
\newcommand{\An}{\vec{A}_n}
\newcommand{\Dn}{\vec{D}_n}
\newcommand{\En}{\vec{E}_n}
\newcommand{\der}{\mathcal{D}^b}
\newcommand{\per}{\operatorname{per}\nolimits}

\begin{document}

\begin{abstract}
We classify the thick subcategories of an algebraic triangulated standard category with finitely many indecomposable objects. 
\end{abstract}

\maketitle

\section{Introduction}
A full subcategory $\mathcal{S}$ of a triangulated category $\mathcal{T}$ is called \emph{thick} if it is a triangulated subcategory of $\mathcal{T}$ which is in addition closed under taking direct summands. 

A classification of thick subcategories can be used to gain structural information about the ambient triangulated category.

The classification problem was approached in various mathematical fields as in stable homotopy theory, in commutative algebra and in the representation theory of finite groups.
The first work was done by Hopkins and Smith concerning the $p$-local finite stable homotopy category \cite{Hopkins}. Hopkins \cite{Hopkins2} and Neeman \cite{Neeman} studied the concept for the category of perfect complexes $\mathcal{D}^{\per}(R)$ of a noetherian ring $R$. They showed that the thick subcategories of $\mathcal{D}^{\per}(R)$ correspond to specialization closed subsets of the prime ideal spectrum of $R$. Benson, Carlson and Rickard \cite{Benson} classified the thick subcategories of the stable module category of the group algebra $kG$ for a $p$-group $G$ in terms of closed subvarietes of the maximal ideal spectrum of the group cohomology ring.

The triangulated categories of our interest are algebraic triangulated categories. A triangulated category is called $\emph{algebraic}$ if it is triangle equivalent to the stable category of a Frobenius category. An example for such a category is the stable module category $\stmod(\Lambda)$ of a self-injective $k$-algebra $\Lambda$ where $k$ is a field. For instance, the group algebra $kG$ for a finite group $G$ is self-injective. 

We concentrate on algebraic triangulated categories which have only finitely many objects and which are standard, i.e.\ the category is equivalent to the mesh category of its Auslander-Reiten quiver. For technical reasons we also assume connectivity. 

The structure of these categories is well understood by work of Amiot \cite{Amiot}. In fact, they are triangle equivalent to the orbit category of the bounded derived category $\der(\mod(k\q{\Delta}))$ of a path algebra of a quiver of Dynkin type $\Delta$ with respect to a specific group of automorphisms of $\der(\mod(k\q{\Delta}))$. From this we define the type of an algebraic triangulated category to be $(\Delta,r,t)$ where $r,t$ depend on the group of automorphisms. 

Moreover,  Ingalls and Thomas \cite{Ingalls} give a combinatorial classification of exact abelian extension-closed subcategories of the abelian category $\mod(k\vec{\Delta})$ in terms of noncrossing partitions associated to $\vec{\Delta}$. 

The poset $\NC_{\vec{\Delta}}$ of noncrossing partitions associated to $\vec{\Delta}$ consists of certain elements of the Coxeter group associated to $\vec{\Delta}$.  For the Dynkin types $A$ and $D$ there is an alternative description $\NC^{A}(n)$ \cite{Reiner} and $\NC^{D}(n)$ \cite{Reiner2} of this poset induced by the hyperplane arrangement of the corresponding root system. $\NC^{A}(n)$ corresponds to the original definition of noncrossing partitions due to Kreweras \cite{Kreweras}.

In this paper, we connect the triangulated and the abelian world and we use the given classification in terms of noncrossing partitions for the classification of the thick subcategories of algebraic triangulated categories. This leads to the main result of this paper. 

\begin{Thm} 
Let $\mathcal{T}$ be a finite triangulated category which is connected, algebraic and standard of type $(\Delta,r,t)$ excluding the cases $(D_n,r,2)$ for $n$ even and $(D_4,r,3)$. Put $s=\gcd(h_\Delta,p)$ where $p$ is a natural number depending on the type of $\mathcal{T}$ and $h_\Delta$ is the Coxeter number. Then, there is a bijective correspondence between
\begin{itemize}
\item the thick subcategories of $\mathcal{T}$, and
\item elements of $\NC_{\vec{\Delta}}$ which are invariant under $s$-fold conjugation by the Coxeter element. 
\end{itemize}
\end{Thm}

Additionally, we formulate this result on the level of the alternative definition of noncrossing partitions $\NC^{A}(n)$ and $\NC^D(n)$. This consideration is combinatorial and we use it to determine the number of thick subcategories in these cases. 

For the two excluded types of the theorem it is not possible to describe the effect of the action of the group of automorphisms of $\der(\mod(k\q{\Delta}))$ in $\NC_{\vec{\Delta}}$. Hence, we classify the type $(D_4,r,3)$ by hand. The remaining type $(D_n,r,t)$ for $n$ even is covered by the classification in terms of certain invariant elements of $\NC^{D}(n)$ mentioned above. 

The paper is organised as follows. First we discuss the structure of finite algebraic triangulated categories. Then we introduce noncrossing partitions and explain the classification of exact abelian extension-closed subcategories. The subsequent section connects the two concepts and contains the main result. Afterwards, we analyse more closely the cases $A$ and $D$. An overview at the end helps to get along with the classification of the different types. 
At the end, we mention an application of the used method to the classification of thick subcategories of other orbit categories, namely cluster categories.   

Finally, a word on notation. If $\vec{\Delta}$ is a quiver, we denote by $\Delta$ its underlying graph. If the object does not depend on the orientation, we will just write $\Delta$ as for the repetition $\mathbb{Z}\Delta$ of the quiver or for the Coxeter number $h_\Delta$. Note that the classification of thick subcategories in terms of $\NC_{\vec{\Delta}}$ requires the choice of an orientation $\vec{\Delta}$. But this is not a problem since we are free to choose any orientation.  

Throughout this paper, let $k$ be an algebraically closed field, algebras are finite-dimensional $k$-algebras, and by $\mod(A)$ we mean the category of finite-dimensional left $A$-modules. All categories are $k$-linear and have finite-dimensional $\Hom$- and $\Ext$-spaces.

\section{The structure of finite algebraic triangulated categories} \label{finite_triangulated_cat}
\begin{Def}
A triangulated category $\mathcal{T}$ is \emph{locally finite} if for each indecomposable $X$ of $\mathcal{T}$ there are at most finitely many isomorphism classes of indecomposables $Y$ such that $\Hom_{\mathcal{T}}(X,Y)\neq 0$. If there are only finitely many indecomposable objects at all the category is called \emph{finite}.
\end{Def}

The structure of a locally finite triangulated category is classified via the shape of its Auslander-Reiten quiver. In order to describe this shape, we recall some definitions and notations concerning translation quivers and automorphism groups of quivers.

Let $\vec{\Delta}=(\vec{\Delta}_0,\vec{\Delta}_1,s,t)$ be a quiver. For a vertex $x\in \vec{\Delta}_0$ we denote by $x^+$ the set of direct successors of $x$ and by $x^-$ the set of direct predecessors of $x$. 

\begin{Def}
Let $\vec{\Delta}$ be a quiver. The \emph{repetition} of $\vec{\Delta}$ \cite{Riedtmann} is the translation quiver $\mathbb{Z}\Delta=\mathbb{Z}\vec{\Delta}$ defined as follows. The vertices of $\mathbb{Z}\Delta$ are the pairs $(i,x)$ with $i\in \mathbb{Z}$ and $x \in \vec{\Delta}_0$. To each arrow $\alpha\colon x \rightarrow y$ in $\vec{\Delta}$ and each $i \in \mathbb{Z}$ there is an arrow $(i,\alpha)\colon(i,x) \rightarrow (i,y)$ and an arrow $\sigma(i,\alpha)\colon(i-1,y) \rightarrow (i,x)$. The \emph{translation} $\tau$ is defined on $(\mathbb{Z}\Delta)_0$ via $\tau(i,x)=(i-1,x)$. The repetitive quiver does not depend on the orientation of $\vec{\Delta}$. 
\end{Def}

\begin{Def}
A group of automorphisms $G$ of a quiver $\vec{\Delta}$ is said to be \emph{admissible} \cite{Riedtmann} if no orbit of $G$ intersects a set of the form $\{x\} \cup x^+ $ or $\{x\} \cup x^-$ in more than one point. It is said to be \emph{weakly admissible} \cite{Dieterich} if, for each $g \in G \setminus \{1\}$ and for each $x \in \vec{\Delta}_0$ we have $x^+ \cap (gx)^+ =\emptyset$. Note that admissible implies weakly admissible. 
\end{Def}

\begin{Thm}[Amiot \cite{Amiot}]
Let $\mathcal{T}$ be a Krull-Schmidt locally finite connected triangulated category. Then, the Auslander-Reiten quiver $\Gamma_{\mathcal{T}}$ of $\mathcal{T}$ is isomorphic to $\mathbb{Z}\Delta / \Phi$ where $\Delta$ is Dynkin of type $A,D$ or $E$ and $\Phi$ is a weakly admissible group of automorphisms of $\mathbb{Z}\Delta$. The underlying graph $\Delta$ is unique up to isomorphism, and the group $\Phi$ is unique up to conjugacy. \qed
\end{Thm}

This generalises Riedtmann's theorem \cite{Riedtmann} which says that the stable Auslander-Reiten quiver of a representation-finite algebra is isomorphic to $\mathbb{Z}\Delta / \Phi$ where $\Delta$ is Dynkin and $\Phi$ is an admissible group of automorphisms. Riedtmann also gives a complete list of possible admissible groups of automorphisms of $\mathbb{Z}\Delta$. Amiot extends this list to a list of possible weakly admissible groups of automorphisms as follows. 

\begin{Thm}[Amiot \cite{Amiot}] \label{wagroups}
Let $\Delta$ be a Dynkin graph and let $\Phi$ be a non trivial weakly admissible group of automorphisms of $\mathbb{Z}\Delta$. This is a list of its possible generators.
\begin{itemize}
\item $\Delta = A_n$ with $n$ odd: possible generators are $\tau^r$ and $\phi \tau^r$ with $r \geq 1$ where $\phi$ is the reflection at the `central line' of $\mathbb{Z}A_n$ which is given by the vertices $\{(i,\frac{n+1}{2}) \mid i \in \mathbb{Z}\}$. Here we take as a basis of $\mathbb{Z}A_n$ the linearly oriented $\An$.\\
\begin{xy}
  \xymatrix@R=2ex@C=4ex{
   &  &  1 \ar[r] & 2 \ar[r]&  \cdots \ar[r] &n-1 \ar[r]   & n &
  }
\end{xy}
\item $\Delta = A_n$ with $n$ even: then possible generators are $\phi \tau^r$ with $r \geq 1$ where $\phi(p,q)=(p+q-\frac{n}{2}-1,n+1-q)$. Here $\phi^2=\tau$. 
\item $\Delta = D_n$ with $n \geq 5$: possible generators are $\tau^r$ and $\phi \tau^r$ with $r\geq 1$ where $\phi$ exchanges $(i,n-1)$ and $(i,n)$ $\forall i \in \mathbb{Z}$ and fixes the other vertices of $\mathbb{Z}D_{n}$. Here we take as a basis of $\mathbb{Z}D_{n}$ the linearly oriented $\Dn$.
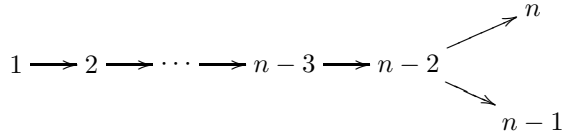
\begin{xy}
  \xymatrix@R=2ex@C=4ex{
   	& &		  & 		&				  &			   &					& n \\
    & & 1 \ar[r] & 2 \ar[r]& \cdots \ar[r]   & n-3 \ar[r] & n-2 \ar[rd] \ar[ru]&    \\
    & &          &			&				  &			   &					& n-1
  }
\end{xy}
\item $\Delta = D_4$: possible generators are $\tau^r$ and $\phi \tau^r$ with $r\geq 1$ and where $\phi$ is either defined as for $n \geq 5$ or as follows.
Again with the linearly oriented $\q{D}_{4}$, $\phi((i,3))=(i,4), \phi((i,4))=(i-1,1), \phi((i-1,1))=(i,3)$ and $\phi((i,2))=(i,2)$ $\forall i \in \mathbb{Z}$.
\item $\Delta = E_6$: possible generators are $\tau^r$ and $\phi \tau^r$ with $r \geq 1$ and where $\phi$ is the reflection at the central line of $\mathbb{Z}E_6$ which is given by $\{(i,3),(i,4) \mid i \in \mathbb{Z}\}$. Here we assume the following orientation and numbering.\\
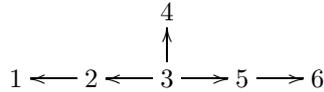
\begin{xy}
  \xymatrix@R=3ex@C=4ex{
   	 & &	  & 	    & 4			             &			  &   &  \\
     & & 1   & 2 \ar[l]& 3 \ar[l] \ar[u] \ar[r] & 5 \ar[r]   & 6 &    
  }
\end{xy}
 
\item $\Delta = E_7,E_8$: possible generators are $\tau^r$ with $r\geq 1$. \qed
\end{itemize} 
\end{Thm}

\begin{Rem}
The only weakly admissible group of automorphisms which is not admissible occurs for $A_n$ with $n$ even and is generated by $\phi$. 
\end{Rem}

\begin{Rem}
If $\mathcal{T}$ is locally finite, but not finite, then the group $\Phi$ of automorphisms is trivial. 
\end{Rem}

Knowing the shape of the Auslander-Reiten quiver does not help us yet to study thick subcategories. We need to know the triangulated structure of $\mathcal{T}$. 

\begin{Thm}[Amiot \cite{Amiot}] \label{alg_triangulated}
If $\mathcal{T}$ is a finite triangulated category which is connected, algebraic and standard, then there exists a Dynkin diagram $\Delta$ of type $A,D$ or $E$ and an auto-equivalence $\Phi$ of $\der(\mod(k\q{\Delta}))$ such that $\mathcal{T}$ is triangle equivalent to the orbit category $\der(\mod(k\q{\Delta}))/\Phi$. \qed
\end{Thm}

\begin{Rem}
Since $\mathcal{T}$ is standard and since by Happel \cite{Happel} the category of indecomposable objects in $\der(\mod(k\q{\Delta}))$ is equivalent to the mesh category $k(\mathbb{Z}\Delta)$, $\Delta$ and $\Phi$ are those induced by $\Delta$ and the group of automorphisms coming from the Auslander-Reiten quiver of $\mathcal{T}$. We denote the automorphism of $\der(\mod(k\q{\Delta}))$ induced by the group of automorphisms $\Phi$ of $\mathbb{Z}\Delta$ by the same character $\Phi$. 

Note that the construction of the orbit category requires an automorphism on $\mathcal{T}$. A standard construction allows one to replace a category with auto-equivalence by a category with automorphism. 
\end{Rem}

Finally, one may ask the question whether each Dynkin type $\Delta$ and each weakly admissible group of automorphisms $\Phi$ give rise to a locally-finite triangulated category with Auslander-Reiten quiver $\mathbb{Z}\Delta/\Phi$. Xiao and Zhu \cite{Xiao} point out that this is actually true. Just take the  orbit category $\der(\mod(k\q{\Delta}))/\Phi$. To show this one checks the assumptions in B. Keller's theorem on triangulated orbit categories \cite{Keller}. 

However, it is not possible to realise each possible Auslander-Reiten quiver $\mathbb{Z}\Delta/\Phi$ via the stable module category of a representation-finite self-injective algebra. Thus, the generalisation to an arbitrary algebraic triangulated category really leads to a greater class of categories.

But which groups do occur in case of the stable module category of a re\-pre\-sentation-finite self-injective standard algebra $\Lambda$? H. Asashiba \cite{Asashiba} classifies them as follows. 

Let $\Lambda$ be a representation-finite self-injective standard algebra with stable Aus\-lander-Reiten quiver $_s\Gamma_\Lambda\cong \mathbb{Z}\Delta/\langle \phi \tau^r \rangle$. Here $\tau$ is the translation on $\mathbb{Z}\Delta$ and $\phi$ is an automorphism of order $1$, $2$ or $3$ defined in \ref{wagroups} such that $\langle \phi \tau^r \rangle$ is an admissible group of automorphisms of $\mathbb{Z}\Delta$. 
Define the type  $\typ(\Lambda)=(T(\Lambda),f(\Lambda),t(\Lambda))$ of $\Lambda$ where $T(\Lambda)=\Delta$ is the tree class, $f(\Lambda)=r/m_{\Delta}$ is the \emph{frequency}, $t(\Lambda)$ is the \emph{order} of $\phi$ and $m_{\Delta}=h_\Delta-1$ where $h_\Delta$ is the \emph{Coxeter number} associated to $\Delta$, i.e.\ the order of the Coxeter element. The type determines $\Lambda$ up to stable equivalence (i.e.\ a triangle equivalence between the stable module categories). 
\begin{Thm}[Asashiba \cite{Asashiba}]
Let $\Lambda, \Lambda'$ representation-finite self-injective algebras. 
\begin{enumerate}
\item If $\Lambda$ is standard and $\Lambda'$ is non-standard, then $\Lambda$ and $\Lambda'$ are not stably equivalent. 
\item If both $\Lambda$ and $\Lambda'$ are standard, then $\Lambda$ and $\Lambda'$ are stably equivalent if and only if $\typ(\Lambda)=\typ(\Lambda')$. 
\item If both $\Lambda$ and $\Lambda'$ are non-standard, then $\Lambda$ and $\Lambda'$ are stably equivalent if and only if $\typ(\Lambda)=\typ(\Lambda')$.  \qed
\end{enumerate}
\end{Thm}

Moreover, Asashiba gives a complete list of these types for standard algebras:
\begin{Thm}[Asashiba \cite{Asashiba}]
 The set of types of standard representation-finite self-injective algebras is equal to the disjoint union of the following sets
\begin{itemize}
 \item $\{(A_n,s/n,1) \mid n,s \in \mathbb{N}\}$;
 \item $\{(A_{2p+1},s,2) \mid p,s \in \mathbb{N}\}$;
 \item $\{(D_n,s,1) \mid n,s \in \mathbb{N},n\geq4\}$;
 \item $\{(D_{3m},s/3,1) \mid m,s \in \mathbb{N}, m\geq2, 3 \nmid s\}$;
 \item $\{(D_n,s,2) \mid n,s \in \mathbb{N},n\geq 4\}$;
 \item $\{(D_4,s,3) \mid s \in \mathbb{N}\}$;
 \item $\{(E_n,s,1) \mid n=6,7,8,s \in \mathbb{N}\}$;
 \item $\{(E_6,s,2) \mid s \in \mathbb{N}\}$. \qed
\end{itemize}
\end{Thm}

Following Asashiba's approach, we define the type of a finite connected triangulated category.
\begin{Def}
Let $\mathcal{T}$ be a finite connected triangulated category with Auslander-Reiten quiver $\Gamma_{\mathcal{T}} \cong \mathbb{Z}\Delta/\Phi$. As we have seen, $\Phi$ is always generated by an element $\phi \tau^r$ with $r \geq 1$ where $\phi$ is an automorphism of order $t=1,2,3$ or of infinite order (which only appears for $A_n$ with $n$ even). Then we define the \emph{type} of $\mathcal{T}$ to be $\typ(\mathcal{T})=(\Delta,r,t)$ where $t\in \{1,2,3,\infty\}$.

Note that the type does not depend on the orientation of $\Delta$. 
\end{Def}

\begin{Rem}
Note that this definition is not the generalisation of the type of a representation-finite self-injective algebra since the second entry is different. E.g.\ if we have an algebra $\Lambda$ with $\typ(\Lambda)=(A_5,2,2)$, then the type of the corresponding triangulated category is given by $\typ(\stmod(\Lambda))=(A_5,10,2)$. 
\end{Rem}

The type of a finite connected triangulated category is contained in the disjoint union of the following sets.
\begin{itemize}
\item $\{(A_n,r,1) \mid n,r \in \mathbb{N} \}$;
\item $\{(A_{2n+1},r,2) \mid n,r \in \mathbb{N} \}$;
\item $\{(A_{2n},r,\infty) \mid n,r \in \mathbb{N} \}$;
\item $\{(D_n,r,1) \mid n,r \in \mathbb{N}, n \geq 4 \}$;
\item $\{(D_n,r,2) \mid n,r \in \mathbb{N}, n \geq 4 \}$;
\item $\{(D_4,r,3) \mid r \in \mathbb{N}\}$;
\item $\{(E_n,r,1) \mid r \in \mathbb{N},n=6,7,8\}$;
\item $\{(E_6,r,2) \mid r \in \mathbb{N}\}$.
\end{itemize}

\section{Wide subcategories and noncrossing partitions} \label{wide_and_ncp}

An exact abelian extension-closed subcategory of an abelian category is called \emph{wide}. Thomas and Ingalls classify the wide subcategories of $\mod(k\vec{\Delta})$ for $\vec{\Delta}$ a Dynkin quiver in terms of noncrossing partitions.
\begin{Def}
Let $\mathcal{A}$ be an abelian category. 
We call an object $M \in \mathcal{A}$ \emph{exceptional} if it has no self-extensions. A sequence of objects $X_{1},\ldots,X_{r}$ in $\mathcal{A}$ is called \emph{exceptional} if each $X_{i}$ is exceptional, and for $i<j$, $\Hom(X_{j},X_{i})=0$ and $\Ext^{1}(X_{j},X_{i})=0$. 

An exceptional sequence $E_{1}, \ldots, E_{r}$ in a wide subcategory $\mathcal{W}\subseteq \mathcal{A}$ is called \emph{complete} if $r$ equals the number of simples in $\mathcal{W}$.
\end{Def}

\begin{Def}
Fix an acyclic finite quiver $\vec{\Delta}$. 
The \emph{Euler form} on $K_0(k\q{\Delta})\cong \mathbb{Z}^n$ is defined by
$$ \langle \dimvec(M),\dimvec(N) \rangle = \dim_k\Hom(M,N) - \dim_k \Ext^1(M,N).  $$
Let $(-,-)$ be the symmetrisation of this, i.e.\ $(v,w)=\langle v,w \rangle + \langle w,v \rangle$ for $v,w \in \mathbb{Z}^n$. 
$v=(v_1,\ldots,v_n)\in\mathbb{Z}^n$ is called a \emph{positive root} if $(v,v)=2$ and $v_j \geq 0$ for all $1\leq j \leq n$.
To any positive root $v$ we associate a \emph{reflection} $s_v \colon \mathbb{Z}^n \rightarrow \mathbb{Z}^n$ by
$$ s_{v}(w)=w-(v,w)v.   $$
The group of automorphisms of $\mathbb{Z}^n$ generated by these reflections is called the \emph{Coxeter group} $W=W_{\q{\Delta}}$ associated to $\q{\Delta}$.

For $w \in W$ let $l(w)$ be the length of the shortest expression for $w$ as a product of reflections. For $u,v \in W$ we say that $u < v$ if $v=ut$ for a reflection $t$ and $l(v)=l(u)+1$. We consider the transitive closure of this relation and denote the resulting partial order by $\leq$.

We may enumerate the vertices of $\vec{\Delta}$ such that the sequence of simple representations $S_{1}, \ldots, S_{n}$ is exceptional. The product of the corresponding reflections $\cox(\vec{\Delta})=s_{1}\ldots s_{n}$ is called \emph{Coxeter element} of $\vec{\Delta}$.

Finally, the \emph{noncrossing partitions} associated to $\vec{\Delta}$ are the elements of the set
$$\NC_{\vec{\Delta}}=\{w \in W \mid \id \leq w \leq \cox(\vec{\Delta})\}.$$
\end{Def}

\begin{Rem}
In case $\Delta=A_n$, $\Delta=B_n$ and $\Delta=D_n$ there is a visualisation of $\NC_{\vec{\Delta}}$ which justifies the name `noncrossing'  and which is very useful for our classification. We will discuss this in the sections \ref{A_n} and \ref{D_n}
\end{Rem}

Now, we have the following combinatorial classification of wide subcategories. 
\begin{Thm}[Ingalls/Thomas \cite{Ingalls}] \label{Ingalls}
Let $\vec{\Delta}$ be Dynkin. Let $\mathcal{C}=\mod(k\vec{\Delta})$. There are bijections between the following objects
\begin{itemize}
 \item wide subcategories in $\mathcal{C}$,
 \item the noncrossing partitions associated to $\vec{\Delta}$.
\end{itemize}
The bijection is given as follows. Let $\mathcal{W}$ be a wide subcategory of $\mod(k\vec{\Delta})$. Take a complete exceptional sequence $E_{1}, \ldots, E_{r}$ of $\mathcal{W}$ and map $\mathcal{W}$ to $$\cox(\mathcal{W})=s_{E_{1}}\ldots s_{E_{r}}.$$ 
For the complete exceptional sequence we may take the sequence of the simple objects of $\mathcal{W}$ in an appropriate order. The map $\cox$ is well defined, no matter which complete exceptional sequence we take.  \qed
\end{Thm}

\begin{Rem} 
We may regard $\mod(k\vec{\Delta})$ as an exact category (in the sense of Quillen \cite{Quillen}) and introduce the concept of a thick subcategory also for these categories. We say that a subcategory is thick provided it is closed under extensions, kernels of epimorphisms, cokernels of monomorphisms and taking direct summands. For the module category of a hereditary algebra one can show that a subcategory is wide if and only if it is thick. (See for example \cite{Dichev}).
\end{Rem}

\section{Thick subcategories of finite algebraic triangulated categories}
Now, we want to connect the abelian with the triangulated world. The crucial link for this is given by Amiot's Theorem \ref{alg_triangulated} mentioned above and a theorem by K. Br\"uning. 

Theorem \ref{alg_triangulated} yields a first correspondence. In order to classify the thick subcategories of an algebraic triangulated category which is triangle equivalent to an orbit category $\der(\mod(k\q{\Delta}))/\Phi$, we study the thick subcategories of the orbit category. 

Thus, it might be useful to consider the thick subcategories of the bounded derived category. This is done by K. Br\"uning.
\begin{Thm}[Br\"uning \cite{Bruening}] \label{Bruening}
Let $\mathcal{A}$ be a hereditary abelian category. The assignments 
$$ f\colon\mathcal{C} \mapsto \{H^{0}C \mid C \in \mathcal{C}\} \ \text{and} \ g\colon\mathcal{M} \mapsto \{C \in \der(\mathcal{A}) \mid H^{n}C\in \mathcal{M} \ \forall n \in \mathbb{Z}\}$$
\noindent
induce mutually inverse bijections between
\begin{itemize}
 \item the class of thick subcategories of (the triangulated category) $\der(\mathcal{A})$, and
 \item the class of wide subcategories in $\mathcal{A}$. \qed
\end{itemize}
\end{Thm}

To complete our chain of correspondences, we finally have to understand the relationship between the thick subcategories of $\der(\mod(k\q{\Delta}))$ and the thick subcategories of the orbit category $\der(\mod(k\q{\Delta}))/\Phi$. Therefore, we have to study the triangulated structure of the orbit category. 

\begin{Def}
Let $\mathcal{T}$ be an additive triangulated category, let $F\colon\mathcal{T} \rightarrow \mathcal{T}$ be an automorphism and let $\langle F \rangle$ be the group of automorphisms generated by $F$. 
The \emph{orbit category} $\mathcal{T}/F=\mathcal{T}/\langle F \rangle$ has the same objects as $\mathcal{T}$ and the morphisms $X \rightarrow Y$ are in bijection to $$\bigoplus_{n\in \mathbb{Z}}\Hom_{\mathcal{T}}(X,F^{n}Y).$$ 
The composition of morphisms is defined in a natural way (compare \cite{Cibils}).
\end{Def}

The triangulated structure of this category is not at all obvious. It is not clear if we can endow it with a triangulated structure such that the canonical projection is a triangle functor. B. Keller gives a positive answer assuming some conditions on the triangulated category which are fulfilled by our orbit category. 

\begin{Thm}[Keller \cite{Keller}] \label{orbitcat}
Let $\mathcal{H}$ be a connected hereditary abelian category admitting a tilting object. Let $\mathcal{T}=\der(\mathcal{H})$, let $F\colon \mathcal{T} \rightarrow \mathcal{T}$ be an automorphism, and let $S\colon \mathcal{T} \rightarrow \mathcal{T}$ be the suspension.
Assume that the following hypotheses hold:
\begin{itemize}
\item For each indecomposable $U$ of $\mathcal{H}$ there are only finitely many integers $i$ such that $F^{i}U$ lies in $\mathcal{H}$.
\item There is an integer $N \geq 0$ such that the $\langle F \rangle$-orbit of each indecomposable of $\mathcal{T}$ contains an object $S^{n}U$ for some $0 \leq n \leq N$ and some indecomposable object $U$ of $\mathcal{H}$.
\end{itemize}
Then, the orbit category $\mathcal{T}/F$ admits a triangulated structure such that the projection $\mathcal{T} \rightarrow \mathcal{T}/F$ is triangulated. 
\end{Thm}
We recall the idea of Keller's proof since we will need some aspects of it later on.
\begin{proof}
The idea is that we embed the orbit category into a bigger category which is triangulated. Then, we extend a morphism in the orbit category to a triangle in the ambient category and show that the extension also lies in the orbit category.

The ambient triangulated category is given by the derived category of a differential graded category. For the basics about dg categories we refer to \cite{Kellerdg}.

Fix a tilting object $T$ of $\mathcal{H}$ and let $A=\End(T)$ be the endomorphism algebra. Then $\der(\mod(A))$ is triangle equivalent to $\der(\mathcal{H})$.
Let $\mathcal{A}=\mathcal{C}_{dg}^{b}(\proj(A))$ be the dg category of bounded complexes of finitely generated projective $A$-modules. We assume that $F$ is a standard equivalence, i.e.\ it is isomorphic to the derived tensor product by a complex of $A$-$A$-bimodules and this defines a dg functor $F$ as well. Let $\mathcal{B}$ be the dg orbit category of $\mathcal{A}$ with respect to $F$. This yields 
$$\mathcal{D}(\Mod(A))\cong \mathcal{D}\mathcal{A} \quad \text{and} \quad  H^{0}\mathcal{B} \cong \der(\mod(A))/F.$$  

Now let $\mathcal{M}$ be the triangulated subcategory of $\mathcal{D}\mathcal{B}$ generated by the representable functors $\mathcal{B}(-,X)=:\hat{X}$ for $X \in \mathcal{B}$. Embed the orbit category 
$$\der(\mod(A))/F \cong H^{0}\mathcal{B}  \hookrightarrow \mathcal{M}\subseteq \mathcal{D}\mathcal{B}, \  X \mapsto \hat{X}.$$

Consider the projection $\pi\colon \mathcal{A} \rightarrow \mathcal{B}$, the restriction $\pi_{\ast}\colon\mathcal{D}\mathcal{B} \rightarrow \mathcal{D}(\Mod (A))$ along $\pi$ and the left adjoint $\pi^{\ast}\colon \mathcal{D}(\Mod (A)) \rightarrow \mathcal{D}\mathcal{B}$ to $\pi_{\ast}$.

$\pi^{\ast}$ restricts to the canonical projection $\der(\mod(A)) \rightarrow \der(\mod(A))/F$. Hence, a morphism in $\der(\mod(A))/F$ is of the form $\pi^{\ast}\hat{X}\rightarrow \pi^{\ast}\hat{Y}$ where $X,Y \in \mathcal{A}$ and $\hat{X}=\mathcal{A}(-,X)$. Extend this to a triangle in $\mathcal{M}$ and apply $\pi_{\ast}$ to this triangle. We get a triangle 
$$ \pi_{\ast}\pi^{\ast}(\hat{X}) \rightarrow \pi_{\ast}\pi^{\ast}(\hat{Y}) \rightarrow \pi_{\ast}(E) \rightarrow S \pi_{\ast}\pi^{\ast}(\hat{X})$$
in $\mathcal{D}(\Mod(A))$. For the elements of the triangle we have
$$\pi_{\ast}\pi^{\ast}(\hat{X})\cong \pi_{\ast}(\mathcal{B}(-,\pi(X)))= \mathcal{B}(\pi(-),\pi(X)) \cong \bigoplus_{n \in \mathbb{Z}}\mathcal{A}(F^{n}-,X)= \bigoplus_{n \in \mathbb{Z}}F^{n}(\hat{X}).$$

By the first assumption of the theorem these objects lie in $\mathcal{D}(\mod(A))$ and so does $\pi_{\ast}(E)$. Using this finiteness property together with the second assumption of the theorem, we can show that $\pi_{\ast}(E)$ is a sum of finitely many $\langle F \rangle$-orbits of shifted indecomposables $Z_{1}, \ldots, Z_{m}$ of $\mathcal{H}$. The adjoint of the inclusion $Z:=\bigoplus_{i=1}^{m}Z_{i} \hookrightarrow \pi_{\ast}(E)$ yields an isomorphism $\pi^{\ast}(Z)\cong E$ where $Z \in \der(\mod(A))$. 
\end{proof}

\begin{Thm} \label{maintheorem}
Let $\mathcal{H}$ be a connected hereditary abelian category admitting a tilting object. Let $\Phi\colon\der(\mathcal{H}) \rightarrow \der(\mathcal{H})$ be an automorphism such that the hypotheses of Keller's theorem hold. Then, the canonical projection $\pi\colon\der(\mathcal{H})\rightarrow \der(\mathcal{H})/\Phi$ induces a bijective correspondence between 
\begin{itemize}
 \item the thick $\langle \Phi \rangle$-invariant subcategories of $\der(\mathcal{H})$, and
\item the thick subcategories of $\der(\mathcal{H})/\Phi$.
\end{itemize}
In particular, this holds for $\mathcal{H}=\mod(k\q{\Delta})$ where $\q{\Delta}$ is a Dynkin quiver, and a group of automorphisms induced by a weakly admissible group of $\mathbb{Z}\Delta$. 
\end{Thm}

\begin{proof}
Let $A=\End(T)$ where $T$ is the tilting object of $\mathcal{H}$.

Let $\mathcal{S}$ be a thick subcategory of $\der(\mod(A))/\Phi$. Then, $\pi^{-1}(\mathcal{S})$ is immediately seen to be thick in $\der(\mod(A))$ since $\pi$ is a triangle functor.

Now let $\mathcal{S}$ be a thick $\langle \Phi \rangle$-invariant subcategory of $\der(\mod(A))$. We want to show that $\pi(\mathcal{S})$ is thick in the orbit category. This direction is not that easy. As seen above it is not obvious how to lift a given triangle 
$$\pi(X) \rightarrow \pi(Y) \rightarrow \pi(Z) \rightarrow S \pi(X)$$
with $\pi(X), \pi(Y) \in \pi(\mathcal{S})$ in $\der(\mod(A))/\Phi$ to a triangle $ X \rightarrow Y \rightarrow Z \rightarrow S X$ in $\der(\mod(A))$. But as in the proof of Keller's Theorem we may apply $\pi_{\ast}$ and obtain a triangle 
$$ \bigoplus_{n \in \mathbb{Z}}\Phi^{n}(X) \rightarrow \bigoplus_{n \in \mathbb{Z}}\Phi^{n}(Y) \rightarrow \bigoplus_{n \in \mathbb{Z}}\Phi^{n}(Z) \rightarrow S \bigoplus_{n \in \mathbb{Z}}\Phi^{n}(X)$$
in $\mathcal{D}(\mod(A))$. This yields a long exact sequence 
\begin{align*}
 \ldots &\rightarrow  H^{p}(\bigoplus_{n \in \mathbb{Z}}\Phi^{n}(X)) \rightarrow H^{p}(\bigoplus_{n \in \mathbb{Z}}\Phi^{n}(Y)) \rightarrow H^{p}(\bigoplus_{n \in \mathbb{Z}}\Phi^{n}(Z)) \\
&\rightarrow H^{p+1}(\bigoplus_{n \in \mathbb{Z}}\Phi^{n}(X)) \rightarrow H^{p+1}(\bigoplus_{n \in \mathbb{Z}}\Phi^{n}(Y)) \rightarrow \ldots
\end{align*}
in $\mod(A)$. Consider the terms of this sequence.
\begin{align*} H^{p}(\bigoplus_{n \in \mathbb{Z}}\Phi^{n}(X))&= H^{0}(S^{p}(\bigoplus_{n \in \mathbb{Z}}\Phi^{n}(X)))
= H^{0}(\bigoplus_{n \in \mathbb{Z}}\Phi^{n}(S^{p}X))\\
&= H^{0}(\bigoplus_{|n|<r}\Phi^{n}(S^{p}X))
\end{align*}
for some $r \in \mathbb{N}$ large enough. We find this $r$ because of the first assumption of Keller's Theorem.
Since $\mathcal{S}$ is thick and $\langle \Phi \rangle$-invariant and $X \in \mathcal{S}$, $\bigoplus_{|n|<r}\Phi^{n}(S^{p}X)$ lies in $\mathcal{S}$ for each $p \in \mathbb{Z}$ and therefore,  $H^{p}(\bigoplus_{n \in \mathbb{Z}}\Phi^{n}(X))$ lies in $H^{0}(\mathcal{S})$ for each $p \in \mathbb{Z}$. 

By Br\"uning's Theorem \ref{Bruening}, $H^{0}(\mathcal{S})$ is a wide subcategory of $\mod(A)$. Hence, the cokernel of $$H^{p}(\bigoplus_{n \in \mathbb{Z}}\Phi^{n}(X)) \rightarrow H^{p}(\bigoplus_{n \in \mathbb{Z}}\Phi^{n}(Y))$$ as well as the kernel of $$H^{p+1}(\bigoplus_{n \in \mathbb{Z}}\Phi^{n}(X)) \rightarrow H^{p+1}(\bigoplus_{n \in \mathbb{Z}}\Phi^{n}(Y))$$ lie in $H^{0}(\mathcal{S})$ for each $p \in \mathbb{Z}$. $H^{p}(\bigoplus_{n \in \mathbb{Z}}\Phi^{n}(Z))$ is an extension of these two objects and therefore, it is also an object of $H^{0}(\mathcal{S})$ for each $p \in \mathbb{Z}$. 

$Z$ is a direct summand of $\bigoplus_{n \in \mathbb{Z}}\Phi^{n}Z$ and hence, $H^{p}(Z)$ is a direct summand of $H^{p}(\bigoplus_{n \in \mathbb{Z}}\Phi^{n}Z)$ for each $p$. Since $H^{0}(\mathcal{S})$ is closed under direct summands, we have that $H^{p}(Z) \in H^{0}(\mathcal{S})$ $\forall p$. Applying Br\"uning's correspondence again, we conclude that $Z \in \mathcal{S}$ and thus $\pi(Z) \in \pi(\mathcal{S})$.
\end{proof}

The theorem raises the question: What are the $\langle \Phi \rangle$-invariant sub\-categories of $\der(\mod(k\q{\Delta}))$ or respectively, which noncrossing partitions are in bijection to them?

Recall that we have a bijective correspondence between the thick subcategories of $\der(\mod(k\q{\Delta}))$ and the noncrossing partitions associated to $\vec{\Delta}$. This correspondence is given by the composition of the correspondences of Br\"uning and Ingalls/Thomas, i.e.\ it is given by the map
\begin{align*}
\tilde{\cox}\colon\{\text{thick subcategories of} \ \der(\mod(k\q{\Delta})) \} &\rightarrow \NC_{\vec{\Delta}}  \\
\mathcal{S} &\mapsto \cox(\mathcal{S} \cap \mod(k\vec{\Delta})). 
\end{align*} 

The translation $\tau$ on $\mathbb{Z}\Delta$ induces an automorphism of $\der(\mod(k\q{\Delta}))$ which is given as follows. Since $\der(\mod(k\q{\Delta})) \cong \mathcal{K}^{b}(\proj(k\q{\Delta}))$, we can understand each element of $\der(\mod(k\q{\Delta}))$ as its projective resolution $$P^{\bullet} \in \mathcal{K}^{b}(\proj(k\q{\Delta})).$$ Then, $\tau(P^{\bullet})=S^{-1}\nu(P^{\bullet})$ where $\nu$ is the Nakayama functor (see \cite{Happel}).

The following Proposition describes how $\tau$ behaves on the noncrossing partitions.

\begin{Prop} \label{tau_action}
Let $\mathcal{S}$ be a thick subcategory of $\der(\mod(k\q{\Delta}))$. Then 
$$\tilde{\cox}(\tau(\mathcal{S}))=\cox(\vec{\Delta}) \tilde{\cox}(\mathcal{S}) \cox(\vec{\Delta})^{-1}$$ 
\end{Prop}
\begin{proof}
Let $E_{1}, \ldots, E_{r}$ be a complete exceptional sequence of simple objects in the wide subcategory $\mathcal{S} \cap \mod(k \vec{\Delta})$. We can reorder this sequence in such way that the possible projective modules lie at the end of the sequence: $$E_{1}, \ldots, E_{s-1}, E_{s}=P(i_{s}), \ldots, E_{r}=P(i_{r}).$$
Just preserve the order within the projectives and the non-projectives, respectively. After that, we only have to take care that $\Hom(P,E)=0$ and $\Ext^{1}(P,E)=0$ for $P$ one of the projective modules and $E$ one of the non-projective modules. The second equation is clear and the first one holds by Schur's lemma.

We show that 
$$\tau(E_{1}), \ldots, \tau(E_{s-1}), I(i_{s}), \ldots, I(i_{r})$$
is a complete exceptional sequence in $\tau(\mathcal{S}) \cap \mod(k\vec{\Delta})$. 

First of all, the sequence should really lie in $\tau(\mathcal{S}) \cap \mod(k\vec{\Delta})$. Let $E$ be one of the non-projective modules of the sequence. We may understand it as a complex concentrated in degree zero. Here, the $\tau$-action corresponds to the usual Auslander-Reiten translation in $\mod(k\vec{\Delta})$ and hence $\tau(E)$ is the complex concentrated in degree zero with entry $\tau(E)$ in degree zero. So $\tau(E) \in \tau(\mathcal{S}) \cap \mod(k\vec{\Delta})$. 

Let $P(i)$ be one of the projective modules of the sequence. Then, $\tau(P(i))=S^{-1}I(i)$ and $\tau(P(i)) \in \tau(\mathcal{S})$. This is thick since $\tau$ is a triangle equivalence on $\der(\mod(k\q{\Delta}))$ and hence $I(i)=S(\tau(P(i))) \in \tau(\mathcal{S})$. Thus, $I(i) \in \tau(\mathcal{S}) \cap \mod(k\vec{\Delta})$. 

Next, we have to check the exceptionality of the sequence. Using the fact that $S$, $\nu$ and therefore also $\tau=S^{-1}\nu$ are triangle equivalences on  $\der(\mod(k\q{\Delta}))=\der(k\q{\Delta})$, we  get for $i<j$
\begin{align*}
\Hom_{k\vec{\Delta}}(\tau(E_{j}),\tau(E_{i}))&=\Hom_{\der(k\q{\Delta})}(\tau(E_{j}), \tau(E_{i}))=\Hom_{\der(k\q{\Delta})}(E_{j},E_{i})=0\\
\Ext^{1}(\tau(E_{j}),\tau(E_{i}))&\cong \Hom_{\der(k\q{\Delta})}(\tau(E_{j}),\tau(SE_{i}))
\cong \Hom_{\der(k\q{\Delta})}(E_{j},SE_{i})\\
&\cong\Ext^{1}(E_{j},E_{i})=0.
\end{align*}
There are no non-zero morphisms from one of the injective objects $I(i)$ to one of the objects $\tau(E_{j})$ since $\tau(E_{j})$ is not injective. Moreover,
\begin{align*}
\Ext^{1}(I(i),\tau(E_{j})) &\cong \Hom_{\der(k\q{\Delta})}(I(i),S\tau(E_{j}))
\cong \Hom_{\der(k\q{\Delta})}(S^{-1}I(i), \tau(E_{j}))\\
&\cong \Hom_{\der(k\q{\Delta})}(\tau^{-1}(S^{-1}I(i)),E_{j})
\cong \Hom_{k\q{\Delta}}(P(i),E_{j})=0.
\end{align*}
And lastly, for $k <l$
$$\Hom_{k\vec{\Delta}}(I(i_{l}),I(i_ {k})) \cong \Hom_{k \vec{\Delta}}(P(i_{l}),P(i_{k}))=0,$$
and clearly, there are no non-trivial extensions between two injectives. 

So we have found an exceptional sequence of length $r$ in $\tau(\mathcal{S}) \cap \mod(k\vec{\Delta})$. Suppose it was not complete. Then, we could find $r+1$ simple objects which build an exceptional sequence. As before we can show that this yields an exceptional sequence $Y_{1}, \ldots, Y_{r+1}$ in $\mathcal{S} \cap \mod(k \vec{\Delta})$, a contradiction since $E_{1}, \ldots, E_{r}$ is complete. 

Now we know that $\cox(\vec{\Delta})(\dimvec(M))=\dimvec(\tau(M))$ for $M \in \mod(k \vec{\Delta})$ indecomposable not projective and $\cox(\vec{\Delta})(\dimvec(P(i)))=-\dimvec(I(i))$ (see for example \cite{Assem}). Hence, $$s_{\tau(E_{i})}=\cox(\vec{\Delta}) s_{E_{i}} \cox(\vec{\Delta})^{-1} \quad \text{and} \quad s_{I(i)}=\cox(\vec{\Delta})s_{P(i)}\cox(\vec{\Delta})^{-1}.$$ 

Altogether, we have
$$\tilde{\cox}(\tau(\mathcal{S}))=s_{\tau(E_{1})}\ldots s_{\tau(E_{s-1})} s_{I(i_{s})} \ldots s_{I(i_{r})}
= \cox(\vec{\Delta})\tilde{\cox}(\mathcal{S})\cox(\vec{\Delta})^{-1}. $$ \end{proof} 

Recall that the automorphism $\Phi$ of $\der(\mod(k\q{\Delta}))$ is induced by a weakly admissible group of automorphisms of $\mathbb{Z}\Delta$ generated by some $\phi\tau^r$ (compare Theorem \ref{wagroups}). The previous Proposition controls the case where the order of $\phi$ equals 1, i.e.\ $\typ(\mathcal{T})=(\Delta,r,1)$, but does not help us in the cases where the order of $\phi$ is $>1$. Fortunately, in most cases one may express $\phi$ in terms of $\tau$ and the shift functor $S$. The following Proposition describes $S$ in terms of $\tau$ and the respective automorphism $\phi$. We conclude this from \cite{Gabriel} and \cite{Bialkowski}.
\begin{Prop} \label{Bialkowski}
Let $\Delta$ be a Dynkin graph, let $h_\Delta=m_\Delta+1$ be the Coxeter number and let $S\colon \der(\mod(k\q{\Delta})) \rightarrow \der(\mod(k\q{\Delta}))$ be the shift functor. 
\begin{enumerate}
\item Let $\Delta$ be of the form $A_{n} (n \geq 3 \text{ odd}), D_{n} (n \text{ odd})$ or $E_{6}$, then $S$ is isomorphic to $\phi \tau^{\frac{h_{\Delta}}{2}}$. Here $\phi$ is the automorphism of $\der(\mod(k\q{\Delta}))$ of order $2$ induced by the automorphism $\phi$ of $\mathbb{Z}\Delta$ defined in Theorem \ref{wagroups}. 
\item Let $\Delta$ be of the form $A_{1}, D_{n} (n \text{ even}), E_{7}$ or $E_{8}$, then $S\cong \tau^{\frac{h_{\Delta}}{2}}$.
\item Let $\Delta$ be of the form $A_{n} (n \text{ even})$, then $S \cong \phi \tau^{\frac{m_{\Delta}}{2}}$ where $\phi$ is the automorphism of infinite order discussed in Theorem \ref{wagroups}.
\end{enumerate}
In particular, in all cases $S^{2}\cong \tau^{h_{\Delta}}$.  \qed
\end{Prop}

Before we apply this to the classification, we insert one Proposition which simplifies later calculations. 

\begin{Prop}\begin{enumerate} \label{simplif}
\item Let $s:=\gcd(h_{\Delta},p)$ with $p \in \mathbb{N}$. Then, a thick subcategory $\mathcal{S}$ of $\der(\mod(k\q{\Delta}))$ is $\langle \tau^{p} \rangle$-invariant if and only if it is $\langle \tau^{s} \rangle$-invariant. 
\item Let $s:= r \mod h_{\Delta}$ with $r \in \mathbb{N}$ and let $\phi \colon \der(\mod(k \q{\Delta})) \rightarrow \der(\mod(k \q{\Delta}))$ be  an automorphism. A thick subcategory $\mathcal{S}$ of $\der(\mod(k\q{\Delta}))$ is $\langle \phi\tau^{r} \rangle$-invariant if and only if it is $\langle \phi \tau^{s} \rangle$-invariant. 
\end{enumerate}
\end{Prop}
\begin{proof}
Let $\gcd(h_{\Delta},p)=s$. Then, there are $m,n \in \mathbb{Z}$ with $s=mh_{\Delta}+np$. Let $\mathcal{S}\neq 0$ be a thick subcategory of $\der(\mod(k\q{\Delta}))$ which is $\langle \tau^{p} \rangle$-invariant. Let $X\neq 0$ be indecomposable in $\mathcal{S}$.  Then, by Proposition \ref{Bialkowski}
$$  X = \tau^{0}(X) = \tau^{-mh_{\Delta}-np+s}(X) \cong \tau^{-np}(S^{-2m})\tau^{s}(X).$$
Then, $S^{-2m}\tau^{s}(X)\cong \tau^{np}(X) \in \mathcal{S}$ since $\mathcal{S}$ is $\langle\tau^{p}\rangle$-invariant. This implies $\tau^{s}(X)\in \mathcal{S}$ since $\mathcal{S}$ is thick. 

Conversely, a $\langle \tau^{s} \rangle$-invariant subcategory is clearly $\langle \tau^{p} \rangle$-invariant. 

The second part works similarly applying $\tau^{h_\Delta}=S^2$. 
\end{proof}

\begin{Rem}
In particular, there are no proper thick  subcategories in case $p$ and $h_{\Delta}$ are coprime. 
\end{Rem}

Finally, we obtain the following classification. 

\begin{Thm} \label{classif}
Let $\mathcal{T}$ be a finite triangulated category which is connected, algebraic and standard of type $(\Delta,r,t)$ excluding the cases $(D_n,r,2)$ for $n$ even and $(D_4,r,3)$. Put 
$$p=\begin{cases}
r & \text{if} \ \typ(\mathcal{T})=(\Delta,r,1),\\
\frac{h_{\Delta}}{2}+r & \text{if} \ \typ(\mathcal{T})=(A_{2n+1},r,2), (D_n,r,2) \ \text{with $n$ odd or} \ (E_6,r,2),\\
\frac{m_\Delta}{2}+r & \text{if} \ \typ(\mathcal{T})=(A_{2n},r,\infty),
\end{cases}$$
and $s=\gcd(h_\Delta,p)$. Then, there is a bijective correspondence between
\begin{itemize}
\item the thick subcategories of $\mathcal{T}$, and
\item the noncrossing partitions $w$ associated to $\vec{\Delta}$ satisfying $$w=\cox(\vec{\Delta})^{s}w\cox(\vec{\Delta})^{-s}.$$
\end{itemize}
\end{Thm}

\begin{proof}
Let $\mathcal{T}$ be as above. $\mathcal{T}$ is triangle equivalent to $\der(\mod(k\q{\Delta}))/\langle \phi \tau^r \rangle$ where $\phi$ and $\tau$ are induced by the respective weakly admissible groups of automorphisms of $\mathbb{Z}\Delta$. 

By Theorem \ref{maintheorem} the thick subcategories of $\mathcal{T}$ are in bijective correspondence to the thick $\langle \phi \tau^r \rangle$-invariant subcategories of $\der(\mod(k\q{\Delta}))$. 

Using Proposition \ref{Bialkowski}, a thick subcategory of $\der(\mod(k\q{\Delta}))$ is $\langle \phi \tau^r \rangle$-invariant if and only if it is $\langle \tau^p \rangle$-invariant. 

By Proposition \ref{simplif}, Theorem \ref{Bruening}, Theorem \ref{Ingalls} and Proposition \ref{tau_action} the thick $\langle \tau^p \rangle$-invariant subcategories of $\der(\mod(k\q{\Delta}))$ are in correspondence to the elements of $\NC_{\vec{\Delta}}$ which are invariant under $s$-fold conjugation by the Coxeter element. 
\end{proof}

\begin{Rem}
The cases $(D_n,r,2)$ for $n$ even and $(D_4,r,3)$ are excluded in this general theorem since in these cases it is not possible to express the automorphism $\phi$ of order $2$ or $3$ in terms of $\tau$ and $S$. We consider these cases separately in Proposition \ref{class_D_2} and Proposition \ref{class_D_4}. 
\end{Rem}

\section{The $A_{n}$-case} \label{A_n}
In the following, we describe the noncrossing partitions of type $\q{A}_n$ which are invariant under $s$-fold conjugation by the Coxeter element.
The following interpretation of $\NC_{\q{A}_n}$ \cite{Reiner} is helpful. In this section we fix a linearly oriented quiver $\vec{A}_n$. \\
\begin{xy}
  \xymatrix{
     &   &   & 1 \ar[r] & 2 \ar[r]&  \cdots \ar[r] &n-1 \ar[r]   & n &
  }
\end{xy}

\noindent
The classification of thick subcategories does not depend on this choice.  

Let $\Pi^{A}(n)$ be the poset of intersection subspaces of the hyperplanes of the root system of type $A_{n-1}$. Let $V$ be the subspace of $\mathbb{R}^{n}$ for which the coordinates sum to $0$. The root system of type $A_{n-1}$ is given by the integer vectors of $V$ of length $\sqrt{2}$. Choose as simple roots $\alpha_i=e_i-e_{i+1}$ for $1\leq i \leq n-1$. Then, the hyperplane arrangement is given by
$$\{x_{i}= x_{j} \mid 1 \leq i < j \leq n \}.   $$
Thus, we can consider the elements of $\Pi^{A}(n)$ as partitions $w$ of the set $[n]:=\{1,\ldots,n\}$ into disjoint blocks. 

Now, we place the numbers $1,2,\ldots,n$ clockwise around a circle in order ($n$ adjacent to $1$) and we draw a chord of the circle between $i$ and $j$ if they are in the same block of $w$ and no other elements strictly between them when going clockwise from $i$ to $j$ around the circle are also in this block. Then, $\NC^{A}(n)$ consists of the elements of $\Pi^{A}(n)$ in which all these chords may be drawn without crossing each other. This corresponds to the original noncrossing partition lattice introduced by Kreweras \cite{Kreweras}. 

In order to understand that $\NC^{A}(n+1)$ is the same as the poset $\NC_{\vec{A}_n}$ of noncrosssing partitions associated to $\vec{A}_n$ defined in section \ref{wide_and_ncp}, we have to study the Coxeter group of $\vec{A}_n$. Identify the simple reflection $s_i$ with the transposition $(i,i+1)$ for $1 \leq i \leq n$ in the symmetric group $\mathcal{S}_{n+1}$. This yields an isomorphism between the Coxeter group associated to $\vec{A}_n$ and $\mathcal{S}_{n+1}$.

\begin{Thm}[Brady \cite{Brady}]
There is a bijection $f\colon \NC_{\vec{A}_n} \rightarrow \NC^{A}(n+1)$. Write $w \in \NC_{\vec{A}_n}\subset \mathcal{S}_{n+1}$ as a product of disjoint cycles. Then, $f(w)$ is the disjoint union of blocks formed by these cycles.  \qed
\end{Thm}

So what does conjugation by the Coxeter element mean for the elements in $\NC^{A}(n)$?

\begin{Prop} \label{class_A}
Let $\mathcal{T}$ be a finite triangulated category which is connected, algebraic and standard of type $(A_n,r,t)$.
Again let $s=\gcd(h_{A_n},p)$ where $p$ depends on the type as discussed in Theorem \ref{classif}. Then, there is a bijective correspondence between
\begin{itemize}
\item the thick subcategories of $\mathcal{T}$, and
\item the elements of $\NC^{A}(n+1)$ invariant under a clockwise rotation by $s\frac{2\pi}{n+1}$.
\end{itemize}
\end{Prop}
\begin{proof}
The Coxeter element is of the form
$$\cox(\vec{A}_{n})=s_{1}\ldots s_{n} = (1,2) \ldots (n,n+1)=(1,2,3,\ldots,n+1). $$
We can write an element of $\NC_{\vec{A}_{n}}\subset \mathcal{S}_{n+1}$ as a product of disjoint cycles. It is sufficient to observe the conjugation by the Coxeter element of one of these cycles,
\begin{align*}(1,2,\ldots,n+1) &(p_{1},p_{2},\ldots,p_{k}) (n+1,n,\ldots,2,1)\\
&= ([p_{1}+1], \ldots,[p_{k}+1]),
\end{align*}
where $[p_{j}+1]=(p_{j}+1) \mod (n+1)$ if $p_{j}+1\neq n+1$ and $[p_{j}+1]=n+1$ otherwise,
and we see that this is just the clockwise rotation. 
\end{proof}

\begin{Rem}
T. Araya \cite{Araya} also observes this phenomenon concerning rotation of `non-crossing spanning trees' which are in correspondence to exceptional sequences of $\der(\mod(k\An))$. 
\end{Rem}

This combinatorial description enables us to determine the number of thick subcategories. There is a quite obvious bijection between the elements of $\NC^{A}(n+1)$ invariant under rotation by $t\frac{2\pi}{n+1}$ with $t$ a non-trivial divisor of $n+1$ and the noncrossing $B_t$-partitions. Analogously to the $A_n$-partitions, these are defined as follows \cite{Reiner}.

$\Pi^B(n)$ is the poset of partitions $w$ of $[\pm n]:=\{1,2,\ldots,n,-1,-2,\ldots,-n\}$ into blocks with the property that for any block $B$ of $w$, its negative $-B$ is also a block of $w$, and there is at most one block (called the zero block, if present) containing both $i$ and $-i$ for some $i$. Place the numbers $$1,2,\ldots,n,-1,-2,\ldots,-n$$ clockwise around a circle in this order and draw as for a $A_n$-partitions chords of the circle according to the blocks of a partition. Then, $\NC^B(n)$ is the subset of $\Pi^B(n)$ for which all these chords may be drawn without crossing each other. 

\begin{Prop}[Reiner \cite{Reiner}] \label{B_correspondence}
Let $s,n+1 \geq 2$ be integers with $s$ a non-trivial divisor of $n+1$. Then, the subset of elements of $\NC^{A}(n+1)$ invariant under rotation by $s\frac{2\pi}{n+1}$ is isomorphic to $\NC^{B}(s)$. \qed
\end{Prop}

\begin{Prop}[Reiner \cite{Reiner}]
The number of elements in $\NC^B(s)$ equals  $\binom{2s}{s}$. \qed
\end{Prop}

\begin{Prop} \label{nb_A}
Let $\mathcal{T}$ be a finite triangulated category which is connected, algebraic and standard of type $(A_n,r,t)$.
Let $s=\gcd(h_{A_n},p)$, where $p=r$, $p=r+\frac{h_{A_n}}{2}$ or $p=r+\frac{m_{A_n}}{2}$ depending on the torsion order (see Theorem \ref{classif}). There is $x\in \mathbb{N}$ with $h_{A_{n}}=xs$.

The number of thick subcategories of $\mathcal{T}$ equals $\binom{2s}{s}$ if $x>1$ and the $s$-th Catalan number $C_{s}$ otherwise.
\end{Prop}
\begin{proof}
Let $x>1$. 
The number of thick subcategories of $\mathcal{T}$ equals the number of elements of $\NC^{A}(n+1)$ invariant under rotation by $s\frac{2\pi}{n+1}$ and this equals $\binom{2s}{s}$.

Let $x=1$, i.e.\ $h_{\Delta}=s$. Hence, the number of thick subcategories equals the number of elements of $\NC^{A}(n+1)$ invariant under rotation by $2\pi$. All elements of $\NC^{A}(n+1)$ are invariant under rotation by $2\pi$ and the number of elements of $\NC^{A}(n+1)$ equals $C_s=\frac{1}{s+1}\binom{2s}{s}$ as Kreweras \cite{Kreweras} shows. 
\end{proof}

We prove the formula on the number of elements independently, since this proof contains an algorithm for the construction of the elements of $\NC^{A}(n+1)$ invariant under rotation by $s\frac{2\pi}{n+1}$. 

\begin{Lemma} \label{lemma_number}
Let $h=xs$ with $x,s \in \mathbb{N}$ and $x>1$. Then, there is a surjective map 
\begin{align*}
 F\colon \{ \text{elements of} \ \NC^{A}(h) \ \text{invariant under rotation by} \ (s\frac{2\pi}{h})\} 
 \rightarrow  \NC^{A}(s).
\end{align*}
Moreover, for each $w \in \NC^{A}(s)$ there are exactly $s+1$ elements in the preimage of $w$ under $F$. Hence, the number of elements of $\NC^{A}(h)$ invariant under rotation by $s\frac{2\pi}{h}$ equals $(s+1)C_{s}=\binom{2s}{s}$.
\end{Lemma}
\begin{proof}
The map $F$ is given as follows. Let $w=w_{1} \dotcup \ldots \dotcup w_{r}$ be an invariant noncrossing partition of $\{1,\ldots, h\}$. Then, $F(w)=F(w_{1}) \cup \ldots \cup F(w_{r})$ with $F(w_{i})=F(\{{p_{1},\ldots,p_{m}}\})=\{[p_{1}],\ldots,[p_{m}]\}$ where $[p_{j}]=p_{j} \mod s$ if $s \nmid p_{j}$ and $[p_{j}]=s$ otherwise.

To a partition in $\NC^{A}(s)$ we construct all partitions in its preimage. 
This construction requires an involution $\alpha$ on $\NC^{A}(n)$ which is due to Simion and Ullman \cite{Simion}. The map $\alpha\colon \NC^{A}(n) \rightarrow \NC^{A}(n)$ is defined as follows. Let $w \in \NC^{A}(n)$. Draw primed numbers $1',\ldots,n'$ clockwise around a circle so that the primed numbers interlace the unprimed numbers $1,\ldots,n$. Then, $\alpha(w)$ is the unique maximal partition of $\{1',\ldots,n'\}$ such that $w \dotcup \alpha(w)$ is a noncrossing partition of $\{1,1',2,2',\ldots,n,n'\}$. Forgetting the primes, we obtain an element $\alpha(w)\in \NC^{A}(n)$.  

Now let $w \in \NC^{A}(s)$. We list all $s+1$ partitions in the preimage of $w$. In particular, this provides a method to construct all $(s\frac{2\pi}{h})$-invariant noncrossing partitions of $\{1,\ldots,h\}$ from the noncrossing partitions of $\{1,\ldots,s\}$. 

Let $w=w_{1} \dotcup \ldots \dotcup w_{r}$ be a noncrossing partition of $\{1,\ldots,s\}$. A block $w_{i}$ of $w$ corresponds in the preimage
\begin{itemize}
\item either to a big block of cardinality $x|w_{i}|$ 
\item or to $x$ disjoint blocks of cardinality $|w_{i}|$.
\end{itemize}
Otherwise the partition in the preimage would be crossing or not invariant. 

To each block $w_{i}$ of $w$ we define a partition in the preimage which contains a big block of cardinality $x|w_{i}|$. All other blocks of this partition are uniquely determined by this. Otherwise the partition would be crossing. This construction yields $r$ partitions in the preimage.

Next we define to each block of $\alpha(w)$ a partition in the preimage of $w$. Let $v$ be a partition in the preimage of $\alpha(w)$ under $F$ containing a big block. There is a partition $u \in \NC^{A}(h)$ with $v=\alpha(u)$. Thus, 
$$F(u)=F(\alpha(v))=\alpha(F(v))=\alpha(\alpha(w))=w,$$ 
and $u$ does not contain a big block since otherwise $\alpha(u) \dotcup u$ would be crossing. Therefore, $u$ is different from the partitions defined in the first step. Moreover, the constructed partitions to different blocks of $\alpha(w)$ are different since the assignment $\alpha$ is bijective. Simion/Ullman show that $\alpha(w)$ has $s-r+1$ blocks and hence, we obtain $s-r+1$ further partitions in the preimage.

Altogether, we have defined $r+s-r+1=s+1$ partitions in the preimage.

It remains to show that there are no further. Let $v \in \NC^{A}(h)$ be a partition in $F^{-1}(w)$. If $v$ contains a big block, we have counted $v$ in the first step. If not, $v$ contains exactly $xr$ blocks, hence $\alpha(v)$ contains $h-(xr)+1$ blocks. Consider $\alpha(v)$. If $\alpha(v)$ contains a big block, we have counted $v$ in the second step. If not, $\alpha(v)$ contains $x(s+1-r)$ blocks.  Thus,
$h-xr+1=x(s+1-r)$ and this implies $x=1$, a contradiction.
\end{proof}

\begin{Ex}
Let $s=2$ and $h=3\cdot2=6$. The noncrossing partitions of $\{1,2\}$ are the following black coloured partitions.\\ 
 \input{picture_ncp1}
 
The elements of $\NC^{A}(h)$ invariant under rotation by $s\frac{2 \pi}{h}$ are given by the following partitions.
First we construct the preimage of the first partition.  \\
 \input{picture_ncp2}\\
And this is the preimage of the second partition. \\
 \input{picture_ncp3}\\
The grey coloured partitions represent the partitions of the primed numbers as used in the proof of Lemma \ref{lemma_number}.  
\end{Ex}

\newpage
\begin{Ex}
Let $\Lambda$ the Nakayama algebra with $_{s}\Gamma_{\Lambda}=\mathbb{Z}A_{5}/\langle \tau^{4} \rangle$. Then, $h_{A_{5}}=6$ and $s=\gcd(4,6)=2$. 

Hence, we are interested in the noncrossing partitions of $\{1,\ldots,6\}$ which are invariant under rotation by $\frac{2}{3}\pi$. They are listed here together with their corresponding thick subcategories of $\stmod(\Lambda)$ indicated by their indecomposables arranged in the Auslander-Reiten quiver  $_{s}\Gamma_{\Lambda}=\mathbb{Z}A_{5}/\langle \tau^{-4} \rangle$. \\
\input{picture_ncp}

\end{Ex}

\section{The $D_{n}$-case} \label{D_n}
Thanks to Reiner and Athanasiadis \cite{Reiner2}, there is a analogue description for the noncrossing partitions associated to a graph of type $D_{n}$. 

In this section we fix a linearly oriented quiver $\vec{D}_n$ with the following numbering. Again the derived category and therefore the classification of thick subcategories does not depend on this. 

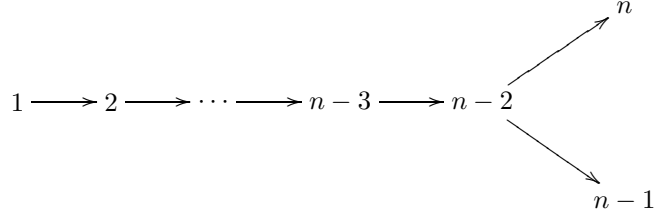
\begin{xy}
  \xymatrix{
   	   &   &	       & 		&				  &			   &					& n \\
       &   &1 \ar[r] & 2 \ar[r]& \cdots \ar[r]   & n-3 \ar[r] & n-2 \ar[rd] \ar[ru]&    \\
       &   &        &			&				  &			   &					& n-1
  }
\end{xy}
 
Let $\Pi^{D}(n)$ the poset of intersection subspaces of the hyperplanes of the root system of type $D_{n}$, i.e.\ the integer vectors in $V=\mathbb{R}^{n}$ of length $\sqrt{2}$. Choose as simple roots $\alpha_{i}=e_{i}-e_{i+1}$ for $1\leq i < n$ and $\alpha_{n}=e_{n-1}+e_{n}$. Then, the hyperplane arrangement is given by
$$\{x_{i}=\pm x_{j} \mid 1 \leq i < j \leq n \}.   $$
Thus, we can consider the elements of $\Pi^{D}(n)$ as partitions $w$ of the set $[\pm n]=\{1,2,\ldots,n,-1,-2,\ldots,-n,\}$ into blocks such that
\begin{itemize}
\item if $B$ is a block, then also $-B$ is a block,
\item there is at most one zero block, i.e.\ a block containing both $i$ and $-i$,
\item the zero block, if present, does not consist of a single pair $\{i,-i\}$.
\end{itemize}
We call those elements \emph{$D_{n}$-partitions}.

Now label the vertices of a regular $(2n-2)$-gon as 
$$1,\ldots,n-1,-1,\ldots,-(n-1)$$
clockwise and label its centroid by $n$ and $-n$. Given a $D_{n}$-partition $w$ and a block $B$ of $w$, let $c(B)$ denote the convex hull of the set of points labeled with the elements of $B$. Two distinct blocks $B$ and $B'$ of $w$ are said to \emph{cross} if $c(B)$ and $c(B')$ do not coincide and one of them contains a point of the other in its relative interior. 

The poset $\NC^{D}(n)$ is defined as the $D_{n}$-partitions $w$ with the property that no two blocks of $w$ cross. 

In order to make the visualisation of the $D_n$-partitions well-defined, we additionally label the convex hull of a non-zero block containing $n$ or $-n$ by $+$ or $-$, respectively. 

The Coxeter group associated to $\Dn$ looks as follows.
Denote by $\mathcal{S}_{2n}$ the symmetric group on $[\pm n]$. For $i \neq -j$ we write $((i,j))=(i,j)(-i,-j)$. By identifying the simple reflections $s_{i}$ in $W_{\vec{D}_{n}}$ with $((i,i+1))$ for $i<n$ and $s_{n}$ with $((-(n-1),n))$, we get that the Coxeter group $W_{\vec{D}_{n}}$ is isomorphic to the subgroup of $\mathcal{S}_{2n}$ generated by the reflections $((i,j))$ for $i\neq-j$. 

For a cycle $c=(i_{1},\ldots,i_{k})$ in $\mathcal{S}_{2n}$ denote by $\bar{c}$ the cycle $(-i_{1},\ldots,-i_{k})$. We call $c\bar{c}$ a \emph{paired cycle} if $c$ and $\bar{c}$ are disjoint whereas a cycle $c=\bar{c}=(i_{1},\ldots,i_{k},-i_{1},\ldots,-i_{k})$ is called a \emph{balanced cycle}. 

One can show that each element of $W_{\vec{D}_{n}}$ is a product of disjoint paired and balanced cycles. 

\begin{Thm}[Athanasiadis/Reiner \cite{Reiner2}]
There is a bijection $$f \colon \NC_{\vec{D}_{n}} \rightarrow \NC^{D}(n).$$ For $w \in \NC_{\vec{D}_{n}}$, $f(w)$ is the partition of $[\pm n]$ 
\begin{itemize}
\item whose nonzero blocks are formed by the paired cycles of $w$ and
\item whose zero block is the union of the elements of all balanced cycles of $w$ if such exist.
\end{itemize}
The inverse $g\colon \NC^{D}(n) \rightarrow \NC_{\vec{D}_{n}}$ maps a partition $w$ to the product of disjoint cycles $g(w)$ 
\begin{itemize}
\item whose paired cycles are formed by the nonzero blocks of $w$, each ordered with respect to the order $-1,-2,\ldots,-n,1,2,\ldots,n$ and
\item whose balanced cycles are $(n,-n)$ and the cycle formed by the entries of the zero block of $w$ other than $n$ and $-n$ ordered as above, if the zero block exist.  \qed
\end{itemize}
\end{Thm}

What does conjugation by the Coxeter element mean for elements of $\NC^{D}(n)$?

\begin{Def}
Let $w \in \NC^{D}(n)$ be the visualisation of a $D_n$-partition. Denote by $\rho\colon \NC^{D}(n) \rightarrow \NC^{D}(n)$ the rotation of $w$ by $\frac{\pi}{n-1}$ and denote by $\sigma\colon \NC^{D}(n) \rightarrow \NC^{D}(n)$ the following operation: If $w \in \NC^{D}(n)$ contains a non-zero block containing $n$ or $-n$, the visualisation of this block is labeled by a sign $+$ or $-$. Then, $\sigma$ changes this sign. On blocks not containing $n$ or $-n$ or on zero-blocks $\sigma$ acts like the identity. 
\end{Def}

\begin{Lemma} \label{cox_D_n}
Let $w \in \NC_{\vec{D}_n}$ and let $f(w)$ be the corresponding element in $\NC^D(n)$. Then,
$$ f(\cox(\vec{D}_n)w\cox(\vec{D}_n)^{-1})= (\sigma \rho)(f(w)).   $$
\end{Lemma}
\begin{proof}
A Coxeter element is given by
$$\cox(\vec{D}_n)=s_1\ldots s_n=(1,2,\ldots,n-1,-1,\ldots,-(n-1))(n,-n).$$ 
$f(w)$ is a disjoint union of blocks corresponding to a disjoint product $w$ of cycles in $\mathcal{S}_{2n}$. Thus, it is sufficient to consider the blocks separately.

If $n$ is not contained in the block, conjugation by the Coxeter element of the corresponding cycle yields as in the $A_n$-case a rotation by $\frac{2\pi}{2(n-1)}=\frac{\pi}{n-1}$. 

If $n$ is contained in a non-zero block, this block corresponds to one part $(i_1,\ldots,i_k)$ of a paired cycle with $i_k=n$ and one can easily check that 
$$\cox(\vec{D}_n)(i_1,\ldots,i_k)\cox(\vec{D}_n)^{-1}$$ is given by the `rotated cycle' in which additionally $n$ is replaced by $-n$. 

If we have a zero-block, this corresponds to the balanced cycles $(n,-n)$ and the cycle formed by the entries of the zero block other than $n$ and $-n$. The latter conjugated by the Coxeter element is again given by the `rotated cycle' and
$$\cox(\vec{D}_n)(n,-n)\cox(\vec{D}_n)^{-1}=(n,-n). $$
Hence, we get the `rotated zero-block'. 
\end{proof}

\begin{Prop} \label{D_n_odd}
Let $\mathcal{T}$ be a finite triangulated category which is connected, algebraic and standard of type $(D_n,r,t)$ excepting $(D_n,r,2)$ for $n$ even and $(D_4,r,3)$. Let $s=\gcd(h_{D_n},r)$ or $s=\gcd(h_{D_n},r+\frac{h_{\Delta}}{2})$ depending on $t$ (see Theorem \ref{classif}). Then, there is a bijective correspondence between
\begin{itemize}
\item the thick subcategories of $\mathcal{T}$, and
\item the elements of $\NC^D(n)$ invariant under $(\sigma \rho)^s$.
\end{itemize}
Note that $\sigma^s=\id$ if $s$ is even and $\sigma^s=\sigma$ if $s$ is odd, and that $\rho$ and $\sigma$ commute.  
\end{Prop}
\begin{proof}
Apply Theorem \ref{classif} and Lemma \ref{cox_D_n}.
\end{proof}

\begin{Ex}
Let $\Lambda$ be a self-injective representation-finite algebra with $\typ(\Lambda)=(D_5,2,2)$, i.e.\ its stable Auslander-Reiten quiver is of the form $\mathbb{Z}D_5/\langle \phi \tau^{14} \rangle$ or $\typ(\stmod(\Lambda))=(D_5,14,2)$. We have $s=\gcd(8,14+\frac{8}{2})=2$ and hence, applying Proposition \ref{D_n_odd}, we look for the $\rho^{2}$-invariant elements of $\NC^{D}(5)$. Here $\rho^{2}$ is a rotation by $\frac{\pi}{2}$. The invariance works for instance for the first two of the following partitions, but not for the third. \\
\input{picture_ncp_Dn_1}\\
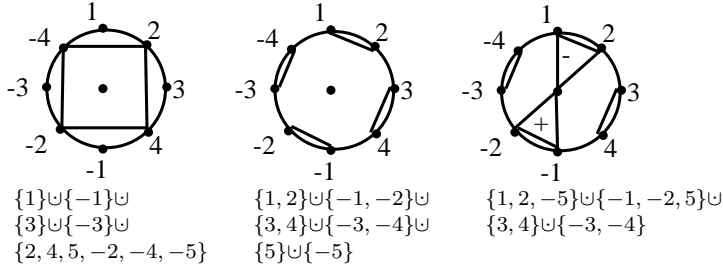
\begin{footnotesize}
\begin{tabular}{p{2.8cm} p{2.7cm} l }
  $\{1\} \dotcup \{-1\} \dotcup$ 	&$\{1,2\} \dotcup \{-1,-2\} \dotcup $  & $\{1,2,-5\} \dotcup \{-1,-2,5\} \dotcup $  \\ 
  $\{3\} \dotcup \{-3\} \dotcup$	&$\{3,4\} \dotcup \{-3,-4\} \dotcup $  & $\{3,4\} \dotcup \{-3,-4\} $ \\  
  $\{2,4,5,-2,-4,-5\}$				&$\{5\} \dotcup \{-5\}$				   &                              \\
  									&									   &
\end{tabular}
\end{footnotesize} 

The corresponding thick subcategories in the first two cases are the following thick subcategories.\\
\input{picture_ncp_Dn_2}
\end{Ex}

Since we cannot express $\phi$ in terms of $S$ and $\tau$ if $n$ is even, we have to study this case separately. The following lemma and the following proposition work for arbitrary $n$ but we only have to use it for even $n$. 

\begin{Lemma} \label{D_n_phi_action}
Let $\mathcal{S}$ be a thick subcategory of $\der(\mod(k\Dn))$. Let
$$\tilde{\cox}\colon\{\text{thick subcategories of} \ \der(\mod(kD_n)) \} \rightarrow \NC^D(n). $$
Let $\phi$ be the automorphism of $\der(\mod(k\Dn))$ induced by the automorphism $\phi$ of order $2$ of $\mathbb{Z}D_n$ defined in section \ref{finite_triangulated_cat}.  Then,
$$\tilde{\cox}(\phi(\mathcal{S}))=\sigma(\tilde{\cox}(\mathcal{S})).    $$
\end{Lemma}

\begin{proof}
Recall that $\phi\colon \mathbb{Z}D_n \rightarrow \mathbb{Z} D_n$ exchanges $(i,n-1)$ and $(i,n)$ $\forall i \in \mathbb{Z}$ and fixes the other vertices. According to Riedtmann \cite{Riedtmann1} we call the vertices $(i,n-1)$ and $(i,n)$ for $i \in \mathbb{Z}$ \emph{high} vertices and the others \emph{low} vertices. Accordingly, we will call the indecomposable elements of $\der(\mod(k\Dn))$ high or low. $\phi$ induces an automorphism of $\der(\mod(k\Dn))$ and hence on $\mod(k\vec{D}_n)$. Determine the Auslander-Reiten quiver of $\mod(k\vec{D}_n)$ with respect to the above orientation of $D_n$ and see that $\phi$ maps the indecomposable high representation $E^{n-1}_{n-t}$ of $\vec{D}_n$ represented by the dimension vector $e_{n-1}+e_{n-2}+\ldots+e_{n-t}$ to the high representation $E^n_{n-t}$ representated by $e_{n}+e_{n-2}+\ldots+e_{n-t}$ and vice versa. The low representations are fixed. 

We have
$$ s_{E^{n}_{n-t}}=s_n s_{n-2} \ldots s_{n-t} \ldots s_{n-2}^{-1} s_{n}^{-1} =((-(n-t),n))$$
and
$$ s_{E^{n-1}_{n-t}}=s_{n-1} s_{n-2} \ldots s_{n-t} \ldots s_{n-2}^{-1} s_{n-1}^{-1}=((n-t,n)).$$
Now let $\mathcal{S}$ be a thick subcategory and let $w$ be the corresponding $D_n$-partition. Again we consider $w$ block by block.

A pair of non-zero blocks $\{i_{1},\ldots,i_{k}\}\dotcup \{-i_{1},\ldots, -i_{k}\}$ corresponds to a paired cycle $(i_1,\ldots,i_k)(-i_1,\ldots,-i_k)$.  This is a product of reflections $$((i_1,i_2))((i_2,i_3))\ldots((i_{k-1},i_k))=s_{E_1}\ldots s_{E_{k-1}}$$
corresponding to an exceptional sequence  $E_1,\ldots,E_{k-1}$ in $\mod(k\vec{D}_n)$. 

If $n$ and $-n$ are not contained in the paired cycle, then none of $E_1,\ldots,E_{k-1}$ are high representations and hence applying $\phi$ does not change $w$. 
If $n$ is contained, we have the product of reflections
$$ (i_1,\ldots,n)(-i_1,\ldots,-n)=((i_1,i_2)) \ldots ((i_{k-1},n))$$
corresponding to an exceptional sequence $E_{i_1},\ldots,E_{i_{k-1}}=E_{i_{k-1}}^{n-1}$. A consideration of the morphisms in $k(\mathbb{Z}D_n)$ as in \cite{Riedtmann1} shows that the sequence $\phi(E_{i_1})=E_{i_1},\ldots,\phi(E_{i_{k-1}}^{n-1})=E_{i_{k-1}}^{n}$ is exceptional. Hence,
$$ s_{\phi(E_1)}\ldots s_{\phi(E_{i_{k-1}})}=((i_1,i_2))\ldots((-i_{k-1},n)) =(i_{1},\ldots,-n)(-i_1,\ldots,n).  $$

Finally, consider a zero-block represented by 
\begin{align*} (i_1,\ldots,i_k,-i_1,\ldots,-i_k)(n,-n)  &= ((i_1,i_2))\ldots((i_{k-1},i_{k}))((i_k,n))((i_k,-n))\\
&=s_{E_{i_1}} \ldots s_{E_{i_k}^{n-1}} s_{E_{i_k}^n}.
\end{align*}
Applying $\phi$ to the exceptional sequence, yields the same term since $((i_{k},n))$ and $((i_k,-n))$ commute. 
\end{proof}

\begin{Prop} \label{class_D_2}
Let $\mathcal{T}$ be a finite triangulated category which is connected, algebraic and standard of type $(D_n,r,2)$. Let $s=r \mod h_{D_n}$. There is a bijective correspondence between
\begin{itemize}
\item the thick subcategories of $\mathcal{T}$, and
\item the elements of $\NC^D(n)$ invariant under $\sigma^{s+1} \rho^s$. 
\end{itemize}
\end{Prop}
\begin{proof}
Use the second part of Proposition \ref{simplif} for the reduction to $s=r \mod h_{D_n}$. Then, apply Lemma \ref{cox_D_n} and Lemma \ref{D_n_phi_action}. 
\end{proof}

The combinatorial description enables us again to determine the number of thick subcategories. We only proof the formulas in case $\typ(\mathcal{T})=(D_n,r,1)$. The other cases work similarly. 

Note that the proof again implies a description of the relevant elements. 

\begin{Prop} \label{nb_D_1}
Let $\mathcal{T}$ be a finite triangulated category which is connected, algebraic and standard of type $(D_n,r,1)$. Let $s=\gcd(h_{D_n}=2n-2,r)$. 

If $s \notin \{2n-2,n-1\}$, the number of thick subcategories equals $\binom{2p}{p}$ where $p=\gcd(n-1,s)$.\\ 
If $s=2n-2$ or if $s=n-1$ with $n$ even, the number of thick subcategories of $\mathcal{T}$ equals $\Cat(D_n)=\binom{2n}{n}-\binom{2n-2}{n-1}$. 

If $s=n-1$ with $n$ odd, the number equals $\Cat(D_{n-1})=\binom{2(n-1)}{n-1}-\binom{2(n-2)}{n-2}$. 
\end{Prop}

\begin{proof}
Denote by $\NC^D_\pm(n)$ the elements of $\NC^D(n)$ with a non-zero block containing $n$ or $-n$. 

Assume $s \notin \{2n-2,n-1\}$. The elements of $\NC^{D}_\pm(n)$ are not $(\sigma \rho)^s$-invariant: Let $w \in \NC^D_\pm(n)$ and let $B$ be a block of $w$ with $n \in B$. Then $C:=(\sigma \rho)^s(B)$ is a block containing $n$ or $-n$ which is neither $B$ itself nor $-B=\sigma \rho^{n-1}(B)$. If the partition was $(\sigma \rho)^s$-invariant, $C$ would be a block of $w$, but this is not possible since $C$ and $B$ or $-B$ are neither equal nor disjoint.

Hence, we have a correspondence between 
\begin{itemize}
\item the thick subcategories of $\mathcal{T}$, and
\item the elements of $\NC^D(n) \setminus \NC^D_\pm(n)$ invariant under rotation by $s\frac{\pi}{n-1}$.
\end{itemize}
Let $p=\gcd(n-1,s)$. Since each element of $\NC^D(n) \setminus \NC^D_\pm(n)$ is invariant under rotation by $(n-1)\frac{\pi}{n-1}$, we can add a correspondence to
\begin{itemize}
\item the elements of $\NC^D(n) \setminus \NC^D_\pm(n)$ invariant under rotation by $p\frac{\pi}{n-1}$. 
\end{itemize}
These elements, in turn, are canonically in correspondence to
\begin{itemize}
\item the elements of $\NC^{A}(2n-2)$ invariant under rotation by $p\frac{\pi}{n-1}$. 
\end{itemize} 
For this purpose, just forget the centroid. Note that a zero-block containing only a single pair $i$ and $-i$ --- which is not allowed for $\NC^D(n)$ --- does not appear since $p\neq n-1$ by assumption. 

Since $p$ is a non-trivial divisor of $2n-2$, by Proposition \ref{B_correspondence} the above elements correspond to
\begin{itemize}
\item $\NC^B(p)$
\end{itemize}
and the number of elements in that is equal to $\binom{2p}{p}$. 

Next, assume $s=2n-2$. Then, the number of thick subcategories equals the number of elements in $\NC^D(n)$ invariant under rotation by $2\pi$, these are all, and by \cite{Reiner2} their number equals the type $D$ Catalan number $\Cat(D_n)$ stated above.  

If $s=n-1$ where $n$ is even, we have to count the elements of $\NC^D(n)$ invariant under $\sigma \rho^{n-1}$, this is again all of $\NC^D(n)$ by definition. 

Finally, let $s=n-1$ and let $n$ be odd. We look for the $(\sigma \rho)^{n-1}=\rho^{n-1}$-invariant elements of $\NC^{D}(n)$. These are the elements $\NC^D(n)\setminus \NC^D_\pm(n)$. As above we compare this with the elements of $\NC^{A}(2n-2)$ which are invariant under rotation by $(n-1)\frac{\pi}{n-1}=\pi$. But this time, this set, which consists of $\binom{2(n-1)}{n-1}$ elements, is too big. We have to ignore the elements which would correspond to a single pair $\{i,-i\}$ in $\NC^D(n)$. With the help of Lemma \ref{lemma_number} we can count these elements. Recall the surjective map
$$F\colon\{\text{elements of} \ \NC^{A}(2n-2) \ \text{inv.\ under rotation by} \ \pi \} \rightarrow \NC^{A}(n-1).   $$ 
Let $w$ be a partition of $\NC^{A}(2n-2)$ which we want to ignore, i.e.\ $w$ contains exactly one block of the form $\{i,i+(n-1)\}$ where $1\leq i \leq n-1$. Then, $F(w)$ contains a singleton $\{i\}$, and $F(w) \setminus \{i\}$ can be understood as an arbitrary element of $\NC^{A}(n-1)$. Hence, the preimage of $F(w)$ consists of $C_{n-2}=\frac{1}{n-1}\binom{2(n-2)}{n-2}$ elements. There are $n-1$ possibilities to place the `forbidden' block in $w$ and thus, we subtract $(n-1)\frac{1}{n-1}\binom{2(n-2)}{n-2}=\binom{2(n-2)}{n-2}$.   
\end{proof}

\begin{Prop} \label{nb_D_2}
Let $\mathcal{T}$ be a finite triangulated category which is connected, algebraic and standard of type $(D_n,r,2)$ with $n$ odd. Let $s=\gcd(2n-2,r+(n-1))$.

If $s \notin \{2n-2,n-1\}$, the number of thick subcategories equals $\binom{2p}{p}$ where $p=\gcd(n-1,s)$.

If $s=2n-2$, the number equals $\Cat(D_n)$. If $s=n-1$, the number equals $\Cat(D_{n-1})$. \qed
\end{Prop}

\begin{Prop} \label{nb_D_3}
Let $\mathcal{T}$ be a finite triangulated category which is connected, algebraic and standard of type $(D_n,r,2)$ with $n$ even. Let $s=r \mod (2n-2)$.

If $s \notin \{0,n-1\}$, the number of thick subcategories equals $\binom{2p}{p}$ where $p=\gcd(n-1,s)$.

If $s=0$ or if $s=n-1$, the number equals $\Cat(D_{n-1})$.  \qed
\end{Prop}

\begin{Ex}
Back to the previous example of a self-injective algebra $\Lambda$ with $\typ(\stmod(\Lambda))=(D_5,14,2)$. We have $s=\gcd(8,18)=2$ and $p=\gcd(4,2)=2$. The $\binom{4}{2}=6$ elements of $\NC^D(5)$ invariant under rotation by $\frac{\pi}{2}$ \\
\input{picture_ncp_Dn_3}\\
correspond to the elements of $\NC^{A}(8)$ invariant under this rotation. These, in turn, can be constructed by the elements of $\NC^{A}(2)$. 
\end{Ex}

Finally, the case of a category of type $(D_4,r,3)$ is missing. Easy observations allow to classify the thick subcategories in this case by hand. 
\begin{Prop} \label{class_D_4}
Let $\mathcal{T}$ be a finite triangulated category which is connected algebraic and standard of type $(D_4,r,3)$.

Put $s=r \mod 3$, hence $s \in \{0,1,2\}$.  

If $s=0$, the only proper thick subcategories of $\mathcal{T}$ are in correspondence to the following thick subcategories of $\der(\mod(k\q{D}_4))$.

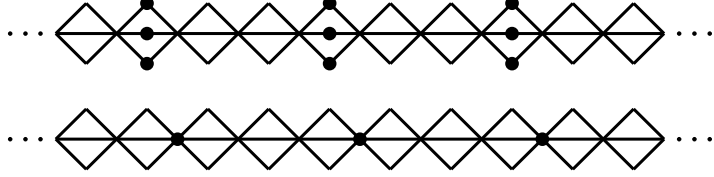
\input{picture_ncp_Dn_4}

In case $s=1$ and $s=2$ there are no proper thick subcategories.
\end{Prop}

\begin{proof}
Let $\phi$ be the automorphism of order $3$ described in Theorem \ref{wagroups}.
Since $S\cong \tau^3$, we have that a thick subcategory of $\der(\mod(k\q{D}_4))$ is $\langle \phi \tau^r \rangle$-invariant if and only if it is $\langle \phi \tau^s \rangle$-invariant. 

In case $s=0$, adding further indecomposables to the described proper thick subcategories would yield the whole category because of thickness. 

The same argument shows that in case $s=1$ and $s=2$ there are no proper thick subcategories. 
\end{proof}

\section{Overview}
Let $\mathcal{T}$ be a finite triangulated category which is connected, algebraic and standard of type $(\Delta,r,t)$. The following table gives an overview of the correspondences to the thick subcategories of $\mathcal{T}$ and of the number of thick subcategories of $\mathcal{T}$. Moreover, you can see where to find the respective information within this paper. 

\newpage
\begin{footnotesize}
\begin{tabular}[t]{|l|p{4cm}|p{3.5cm}|l|}
\hline  type			&  classifying	partitions     &  alternative description  	& number of   \\ 
						&							   &							& partitions  \\
\hline
\hline  $(A_n,r,1)$ 	&  $w \in \NC_{\An}$  with     & elements of $\NC^{A}(n+1)$ 	&   $C_s$ if $s=n+1$;        \\
 						&  $w=\cox(\An)^s w \cox(\An)^{-s}$,& invariant under rotation  	&	$\binom{2s}{s}$ else;\\
 						&  								    & by $s\frac{2\pi}{n+1}$,   	&           				\\
 						&  $s=\gcd(n+1,r)$,					& $s=\gcd(n+1,r)$,				&           				\\
 						&  see Theorem \ref{classif}		& see Proposition \ref{class_A}	&   Proposition \ref{nb_A}   \\
\hline  $(A_{n},r,2)$   &  $w \in \NC_{\An}$  with   		& elements of $\NC^{A}(n+1)$ 	&   $C_s$ if $s=n+1$;        \\
 		$n\geq 3$ odd	& $w=\cox(\An)^s w \cox(\An)^{-s}$,& invariant under rotation  	&	$\binom{2s}{s}$ else; 	\\
 						&  								    & by $s\frac{2\pi}{n+1}$,   	&           				\\
 						&  $s=\gcd(n+1,\frac{n+1}{2}+r)$,	& $s=\gcd(n+1,\frac{n+1}{2}+r)$,&           				\\
 						&  see Theorem \ref{classif}		& see Proposition \ref{class_A}	&   Proposition \ref{nb_A}   \\ 		\hline  $(A_{n},r,\infty)$ &  $w \in \NC_{\An}$  with   	& elements of $\NC^{A}(n+1)$ 	&   $C_s$ if $s=n+1$;        \\
 		$n$ even		&  $w=\cox(\An)^s w \cox(\An)^{-s}$,& invariant under rotation  	&	$\binom{2s}{s}$ else; 	\\
 						&  								    & by $s\frac{2\pi}{n+1}$,   	&           				\\
 						&  $s=\gcd(n+1,\frac{n}{2}+r)$,		& $s=\gcd(n+1,\frac{n}{2}+r)$,  &           				\\
 						&  see Theorem \ref{classif}		& see Proposition \ref{class_A}	&   Proposition \ref{nb_A}   \\ 
\hline  $(D_n,r,1)$ 	&  $w \in \NC_{\Dn}$  with   		& elements of $\NC^{D}(n)$	 	&  $\Cat(D_n)$              \\
						&	$w=\cox(\Dn)^s w \cox(\Dn)^{-s}$,&invariant under $(\sigma\rho)^s$&   if $s=2n-2$ 			\\
 						&  	$s=\gcd(2n-2,r)$,				& $s=\gcd(2n-2,r)$,				&		or $s=n-1$ odd;  	\\
 						&  									& 								&  $\Cat(D_{n-1})$          \\
 						&									&								&  if $s=n-1$ even;          \\
 						&									&								& $\binom{2p}{p}$ else where  \\
 						&									&								& $p=\gcd(n-1,s)$;           \\
 						&  see Theorem \ref{classif}		& see Proposition \ref{D_n_odd}	&   Proposition \ref{nb_D_1} \\	
\hline  $(D_n,r,2)$ 	&  $w \in \NC_{\Dn}$  with   		& elements of $\NC^{D}(n)$	 	&  $\Cat(D_n)$ 		      \\
 		$n$ odd			&  $w=\cox(\Dn)^s w \cox(\Dn)^{-s}$,&invariant under $(\sigma\rho)^s$&	if $s=2n-2$;        	\\
 						&  $s=\gcd(2n-2,\frac{2n-2}{2}+r)$,	& $s=\gcd(2n-2,\frac{2n-2}{2}+r)$,&  $\Cat(D_{n-1})$ 		\\
 						&									&								& if $s=n-1$; 				\\
 						&									&								& $\binom{2p}{p}$ else where \\
 						&									&								& $p=\gcd(n-1,s)$;		\\
 						&  see Theorem \ref{classif}		& see Proposition \ref{D_n_odd}	&   Proposition \ref{nb_D_2} \\		
\hline  $(D_n,r,2)$ 	&  									& elements of $\NC^{D}(n)$	 	&  $\Cat(D_{n-1})$			\\
 		$n$ even		&  									&invariant under $\sigma^{s+1}\rho^s$,&	if $s=0$ 			 \\
 						&								    & $s=r \mod (2n-2)$,			& or $s=n-1$; 					\\
 						&  									& 								&   $\binom{2p}{p}$ else where \\
 						&									&								& $p=\gcd(n-1,s)$;			\\
 						&  									& see Proposition \ref{class_D_2}&   Proposition \ref{nb_D_3} \\
\hline  $(D_4,r,3)$ 	&  									& $s=r \mod 3$,        			&   					      \\	
						&									& $s=0$: six distinguished      &    $8$                 		\\
						&									& proper thick subcategories,   &								\\
						&									& $s=1,2$: no proper ones, 		&    $2$						\\
						&									& see Proposition \ref{class_D_4} &                              \\
\hline $(E_n,r,1)$		&  $w \in \NC_{\En}$ with			&								&								\\
	$n=6,7,8$			& $w=\cox(\En)^sw\cox(\En)^{-s}$,	&								&								\\
						& $s=\gcd(h_{\En},r)$, 				&								&								\\
						& see Theorem \ref{classif}			&								&								\\
\hline $(E_6,r,2)$		&  $w \in \NC_{\q{E}_6}$ with		&								&								\\
						& $w=\cox(\q{E}_6)^sw\cox(\q{E}_6)^{-s}$,	&						&								\\
						& $s=\gcd(12,r+6)$,		 			&								&								\\
						& see Theorem \ref{classif}			&								&								\\
\hline						
		
\end{tabular} 
\end{footnotesize}

\section{Application to cluster categories }

Theorem \ref{maintheorem} applies to cluster categories. 

Let $\vec{\Delta}$ be a Dynkin quiver of type $A$, $D$ or $E$. 
The \emph{cluster category} $\mathcal{C}(k\vec{\Delta})$ associated to $\vec{\Delta}$ is by definition \cite{Buan} the orbit category
$$\mathcal{C}=\mathcal{C}(k\vec{\Delta})=\der(\mod(k\vec{\Delta}))/\tau^{-1}\circ S. $$

By Theorem \ref{orbitcat}, the cluster category is triangulated. Hence, we ask for the thick subcategories.

\begin{Thm}
The cluster category $\mathcal{C}(k\vec{\Delta})$ admits no proper thick subcategories.
\end{Thm}
\begin{proof}
A thick subcategory $\mathcal{S}$ of $\der(\mod(k\vec{\Delta}))$ is invariant under $\tau^{-1} \circ S$ if and only if it is invariant under $\tau^{-1}$. But then, an indecomposable non-zero object $X \in \mathcal{S}$ and its $\tau$-translates generate the whole category $\der(\mod(k\vec{\Delta}))$. 
\end{proof}

\begin{Rem}
The same result holds for the \emph{generalised cluster category} \cite{Keller}
$$ \mathcal{C}^m(k\q{\Delta})= \der(\mod(k\q{\Delta}))/\tau^{-1}\circ S^m  $$
where  $m \in \mathbb{Z}$. 
\end{Rem}

\section*{Acknowledgements}
I have created important parts of this paper during my research stay in Paris with Bernhard Keller in spring 2010. I sincerely thank him for giving me this opportunity, and for his very profitable assistance during this time.

Many thanks to my advisor Henning Krause for his support, for helpful discussions and for coming up with the interesting topic. 

I am also grateful to Hideto Asashiba and to Jan Stovicek for reading this paper in advance and for useful discussions.

\nocite{*}
\bibliographystyle{plain}
\bibliography{Bibliography1}

\end{document}

%% file: picture_ncp1.tex
\scalebox{1} 
{
\begin{pspicture}(0,-0.83171874)(3.2153125,0.83171874)
\definecolor{color15}{rgb}{0.6,0.6,0.6}
\psellipse[linewidth=0.04,dimen=outer](0.693125,-0.01171875)(0.4,0.4)
\psdots[dotsize=0.12](0.693125,0.38828126)
\psdots[dotsize=0.12](0.693125,-0.41171876)
\usefont{T1}{ptm}{m}{n}
\rput(0.61375,0.61828125){1}
\usefont{T1}{ptm}{m}{n}
\rput(0.6689063,-0.68171877){2}
\psellipse[linewidth=0.04,dimen=outer](2.493125,0.00828125)(0.4,0.4)
\psdots[dotsize=0.12](2.493125,0.40828124)
\psdots[dotsize=0.12](2.493125,-0.39171875)
\usefont{T1}{ptm}{m}{n}
\rput(2.4137502,0.65828127){1}
\usefont{T1}{ptm}{m}{n}
\rput(2.4689062,-0.6617187){2}
\psline[linewidth=0.04cm](2.493125,0.38828126)(2.493125,-0.41171876)
\psdots[dotsize=0.12,linecolor=color15](1.093125,-0.02828126)
\psdots[dotsize=0.12,linecolor=color15](0.293125,-0.02828126)
\psdots[dotsize=0.12,linecolor=color15](2.893125,-0.02828126)
\psdots[dotsize=0.12,linecolor=color15](2.093125,-0.02828126)
\usefont{T1}{ptm}{m}{n}
\rput(1.2896875,-0.04328126){\scriptsize \color{color15}1'}
\usefont{T1}{ptm}{m}{n}
\rput(0.0803125,-0.04328126){\scriptsize \color{color15}2'}
\usefont{T1}{ptm}{m}{n}
\rput(1.9003125,-0.04328126){\scriptsize \color{color15}2'}
\usefont{T1}{ptm}{m}{n}
\rput(3.0696874,-0.04328126){\scriptsize \color{color15}1'}
\psline[linewidth=0.04cm,linecolor=color15](0.281875,-0.02828123)(1.101875,-0.02828123)
\end{pspicture} 
}

%% file: picture_ncp2.tex
\scalebox{1} 
{
\begin{pspicture}(0,-1.2117187)(7.4196873,1.2117187)
\definecolor{color204}{rgb}{0.6,0.6,0.6}
\psellipse[linewidth=0.04,dimen=outer](1.1084375,-0.03171875)(0.8,0.8)
\psdots[dotsize=0.12](1.1084375,0.7682813)
\psdots[dotsize=0.12](1.1084375,-0.83171874)
\psdots[dotsize=0.12](1.7284375,0.42828125)
\psdots[dotsize=0.12](1.7684375,-0.45171875)
\psdots[dotsize=0.12](0.4684375,-0.47171873)
\psdots[dotsize=0.12](0.4484375,0.40828127)
\usefont{T1}{ptm}{m}{n}
\rput(1.0121875,1.0182812){1}
\usefont{T1}{ptm}{m}{n}
\rput(1.8756249,0.5782813){2}
\usefont{T1}{ptm}{m}{n}
\rput(1.89375,-0.6217188){3}
\usefont{T1}{ptm}{m}{n}
\rput(1.0803125,-1.0617187){4}
\usefont{T1}{ptm}{m}{n}
\rput(0.2575,-0.6617187){5}
\usefont{T1}{ptm}{m}{n}
\rput(0.2290625,0.53828126){6}
\usefont{T1}{ptm}{m}{n}
\rput(1.5315624,0.8767187){\tiny \color{color204}1'}
\psdots[dotsize=0.12,linecolor=color204](1.449375,0.6517187)
\psdots[dotsize=0.12,linecolor=color204](1.869375,-0.0082813)
\psdots[dotsize=0.12,linecolor=color204](1.489375,-0.6882813)
\psdots[dotsize=0.12,linecolor=color204](0.709375,-0.7082813)
\psdots[dotsize=0.12,linecolor=color204](0.309375,-0.0082813)
\psdots[dotsize=0.12,linecolor=color204](0.729375,0.6517187)
\usefont{T1}{ptm}{m}{n}
\rput(2.0621874,-0.0232813){\tiny \color{color204}2'}
\usefont{T1}{ptm}{m}{n}
\rput(1.58,-0.9632813){\tiny \color{color204}3'}
\usefont{T1}{ptm}{m}{n}
\rput(0.5828125,-0.9632813){\tiny \color{color204}4'}
\usefont{T1}{ptm}{m}{n}
\rput(0.05921875,-0.0432813){\tiny \color{color204}5'}
\usefont{T1}{ptm}{m}{n}
\rput(0.57984376,0.8567187){\tiny \color{color204}6'}
\psellipse[linewidth=0.04,dimen=outer](3.7284374,-0.01171875)(0.8,0.8)
\psdots[dotsize=0.12](3.7284374,0.7882813)
\psdots[dotsize=0.12](3.7284374,-0.81171876)
\psdots[dotsize=0.12](4.3484373,0.44828123)
\psdots[dotsize=0.12](4.3884373,-0.43171874)
\psdots[dotsize=0.12](3.0884376,-0.45171875)
\psdots[dotsize=0.12](3.0684376,0.42828125)
\usefont{T1}{ptm}{m}{n}
\rput(3.6321876,1.0382812){1}
\usefont{T1}{ptm}{m}{n}
\rput(4.495625,0.5982813){2}
\usefont{T1}{ptm}{m}{n}
\rput(4.51375,-0.6017188){3}
\usefont{T1}{ptm}{m}{n}
\rput(3.7003126,-1.0417187){4}
\usefont{T1}{ptm}{m}{n}
\rput(2.8775,-0.64171875){5}
\usefont{T1}{ptm}{m}{n}
\rput(2.8490624,0.55828124){6}
\usefont{T1}{ptm}{m}{n}
\rput(4.1515627,0.8967187){\tiny \color{color204}1'}
\psdots[dotsize=0.12,linecolor=color204](4.069375,0.6717187)
\psdots[dotsize=0.12,linecolor=color204](4.489375,0.0117187)
\psdots[dotsize=0.12,linecolor=color204](4.109375,-0.6682813)
\psdots[dotsize=0.12,linecolor=color204](3.329375,-0.6882813)
\psdots[dotsize=0.12,linecolor=color204](2.929375,0.0117187)
\psdots[dotsize=0.12,linecolor=color204](3.349375,0.6717187)
\usefont{T1}{ptm}{m}{n}
\rput(4.6821876,-0.0032813){\tiny \color{color204}2'}
\usefont{T1}{ptm}{m}{n}
\rput(4.2,-0.9432813){\tiny \color{color204}3'}
\usefont{T1}{ptm}{m}{n}
\rput(3.2028124,-0.9432813){\tiny \color{color204}4'}
\usefont{T1}{ptm}{m}{n}
\rput(2.6792188,-0.0232813){\tiny \color{color204}5'}
\usefont{T1}{ptm}{m}{n}
\rput(3.1998436,0.8767187){\tiny \color{color204}6'}
\psellipse[linewidth=0.04,dimen=outer](6.3484373,-0.01171875)(0.8,0.8)
\psdots[dotsize=0.12](6.3484373,0.7882813)
\psdots[dotsize=0.12](6.3484373,-0.81171876)
\psdots[dotsize=0.12](6.9684377,0.44828123)
\psdots[dotsize=0.12](7.0084376,-0.43171874)
\psdots[dotsize=0.12](5.7084374,-0.45171875)
\psdots[dotsize=0.12](5.6884375,0.42828125)
\usefont{T1}{ptm}{m}{n}
\rput(6.2521877,1.0382812){1}
\usefont{T1}{ptm}{m}{n}
\rput(7.115625,0.5982813){2}
\usefont{T1}{ptm}{m}{n}
\rput(7.13375,-0.6017188){3}
\usefont{T1}{ptm}{m}{n}
\rput(6.3203125,-1.0417187){4}
\usefont{T1}{ptm}{m}{n}
\rput(5.4975,-0.64171875){5}
\usefont{T1}{ptm}{m}{n}
\rput(5.4690623,0.55828124){6}
\usefont{T1}{ptm}{m}{n}
\rput(6.7715626,0.8967187){\tiny \color{color204}1'}
\psdots[dotsize=0.12,linecolor=color204](6.689375,0.6717187)
\psdots[dotsize=0.12,linecolor=color204](7.109375,0.0117187)
\psdots[dotsize=0.12,linecolor=color204](6.729375,-0.6682813)
\psdots[dotsize=0.12,linecolor=color204](5.949375,-0.6882813)
\psdots[dotsize=0.12,linecolor=color204](5.549375,0.0117187)
\psdots[dotsize=0.12,linecolor=color204](5.969375,0.6717187)
\usefont{T1}{ptm}{m}{n}
\rput(7.3021874,-0.0032813){\tiny \color{color204}2'}
\usefont{T1}{ptm}{m}{n}
\rput(6.82,-0.9432813){\tiny \color{color204}3'}
\usefont{T1}{ptm}{m}{n}
\rput(5.8228126,-0.9432813){\tiny \color{color204}4'}
\usefont{T1}{ptm}{m}{n}
\rput(5.2992187,-0.0232813){\tiny \color{color204}5'}
\usefont{T1}{ptm}{m}{n}
\rput(5.819844,0.8767187){\tiny \color{color204}6'}
\psline[linewidth=0.04,linecolor=color204](1.849375,0.0117187)(1.849375,0.0117187)(1.429375,0.6517187)(0.749375,0.6317187)(0.349375,-0.0082813)(0.729375,-0.7082813)(1.449375,-0.6882813)(1.849375,-0.0282813)(1.869375,0.0117187)
\psline[linewidth=0.04cm,linecolor=color204](4.049375,0.6717187)(4.449375,0.0117187)
\psline[linewidth=0.04cm,linecolor=color204](3.329375,-0.6682813)(4.089375,-0.6682813)
\psline[linewidth=0.04cm,linecolor=color204](2.949375,0.0317187)(3.369375,0.6517187)
\psline[linewidth=0.04cm,linecolor=color204](5.969375,0.6517187)(6.649375,0.6517187)
\psline[linewidth=0.04cm,linecolor=color204](7.069375,0.0117187)(6.709375,-0.6682813)
\psline[linewidth=0.04cm,linecolor=color204](5.949375,-0.6882813)(5.569375,0.0117187)
\psline[linewidth=0.04cm](3.709375,0.7917187)(4.389375,-0.4282813)
\psline[linewidth=0.04cm](4.389375,-0.4282813)(3.069375,-0.4282813)
\psline[linewidth=0.04cm](3.089375,-0.4082813)(3.729375,0.7917187)
\psline[linewidth=0.04cm](6.969375,0.4517187)(6.349375,-0.8082813)
\psline[linewidth=0.04cm](6.349375,-0.8082813)(5.709375,0.4317187)
\psline[linewidth=0.04cm](5.709375,0.4317187)(6.949375,0.4317187)
\end{pspicture} 
}

%% file: picture_ncp3.tex
\scalebox{1} 
{
\begin{pspicture}(0,-1.2117187)(7.4196873,1.2117187)
\definecolor{color204}{rgb}{0.6,0.6,0.6}
\psellipse[linewidth=0.04,dimen=outer](1.1084375,-0.03171875)(0.8,0.8)
\psdots[dotsize=0.12](1.1084375,0.7682813)
\psdots[dotsize=0.12](1.1084375,-0.83171874)
\psdots[dotsize=0.12](1.7284375,0.42828125)
\psdots[dotsize=0.12](1.7684375,-0.45171875)
\psdots[dotsize=0.12](0.4684375,-0.47171873)
\psdots[dotsize=0.12](0.4484375,0.40828127)
\usefont{T1}{ptm}{m}{n}
\rput(1.0121875,1.0182812){1}
\usefont{T1}{ptm}{m}{n}
\rput(1.8756249,0.5782813){2}
\usefont{T1}{ptm}{m}{n}
\rput(1.89375,-0.6217188){3}
\usefont{T1}{ptm}{m}{n}
\rput(1.0803125,-1.0617187){4}
\usefont{T1}{ptm}{m}{n}
\rput(0.2575,-0.6617187){5}
\usefont{T1}{ptm}{m}{n}
\rput(0.2290625,0.53828126){6}
\usefont{T1}{ptm}{m}{n}
\rput(1.5315624,0.8767187){\tiny \color{color204}1'}
\psdots[dotsize=0.12,linecolor=color204](1.449375,0.6517187)
\psdots[dotsize=0.12,linecolor=color204](1.869375,-0.0082813)
\psdots[dotsize=0.12,linecolor=color204](1.489375,-0.6882813)
\psdots[dotsize=0.12,linecolor=color204](0.709375,-0.7082813)
\psdots[dotsize=0.12,linecolor=color204](0.309375,-0.0082813)
\psdots[dotsize=0.12,linecolor=color204](0.729375,0.6517187)
\usefont{T1}{ptm}{m}{n}
\rput(2.0621874,-0.0232813){\tiny \color{color204}2'}
\usefont{T1}{ptm}{m}{n}
\rput(1.58,-0.9632813){\tiny \color{color204}3'}
\usefont{T1}{ptm}{m}{n}
\rput(0.5828125,-0.9632813){\tiny \color{color204}4'}
\usefont{T1}{ptm}{m}{n}
\rput(0.05921875,-0.0432813){\tiny \color{color204}5'}
\usefont{T1}{ptm}{m}{n}
\rput(0.57984376,0.8567187){\tiny \color{color204}6'}
\psellipse[linewidth=0.04,dimen=outer](3.7284374,-0.01171875)(0.8,0.8)
\psdots[dotsize=0.12](3.7284374,0.7882813)
\psdots[dotsize=0.12](3.7284374,-0.81171876)
\psdots[dotsize=0.12](4.3484373,0.44828123)
\psdots[dotsize=0.12](4.3884373,-0.43171874)
\psdots[dotsize=0.12](3.0884376,-0.45171875)
\psdots[dotsize=0.12](3.0684376,0.42828125)
\usefont{T1}{ptm}{m}{n}
\rput(3.6321876,1.0382812){1}
\usefont{T1}{ptm}{m}{n}
\rput(4.495625,0.5982813){2}
\usefont{T1}{ptm}{m}{n}
\rput(4.51375,-0.6017188){3}
\usefont{T1}{ptm}{m}{n}
\rput(3.7003126,-1.0417187){4}
\usefont{T1}{ptm}{m}{n}
\rput(2.8775,-0.64171875){5}
\usefont{T1}{ptm}{m}{n}
\rput(2.8490624,0.55828124){6}
\usefont{T1}{ptm}{m}{n}
\rput(4.1515627,0.8967187){\tiny \color{color204}1'}
\psdots[dotsize=0.12,linecolor=color204](4.069375,0.6717187)
\psdots[dotsize=0.12,linecolor=color204](4.489375,0.0117187)
\psdots[dotsize=0.12,linecolor=color204](4.109375,-0.6682813)
\psdots[dotsize=0.12,linecolor=color204](3.329375,-0.6882813)
\psdots[dotsize=0.12,linecolor=color204](2.929375,0.0117187)
\psdots[dotsize=0.12,linecolor=color204](3.349375,0.6717187)
\usefont{T1}{ptm}{m}{n}
\rput(4.6821876,-0.0032813){\tiny \color{color204}2'}
\usefont{T1}{ptm}{m}{n}
\rput(4.2,-0.9432813){\tiny \color{color204}3'}
\usefont{T1}{ptm}{m}{n}
\rput(3.2028124,-0.9432813){\tiny \color{color204}4'}
\usefont{T1}{ptm}{m}{n}
\rput(2.6792188,-0.0232813){\tiny \color{color204}5'}
\usefont{T1}{ptm}{m}{n}
\rput(3.1998436,0.8767187){\tiny \color{color204}6'}
\psellipse[linewidth=0.04,dimen=outer](6.3484373,-0.01171875)(0.8,0.8)
\psdots[dotsize=0.12](6.3484373,0.7882813)
\psdots[dotsize=0.12](6.3484373,-0.81171876)
\psdots[dotsize=0.12](6.9684377,0.44828123)
\psdots[dotsize=0.12](7.0084376,-0.43171874)
\psdots[dotsize=0.12](5.7084374,-0.45171875)
\psdots[dotsize=0.12](5.6884375,0.42828125)
\usefont{T1}{ptm}{m}{n}
\rput(6.2521877,1.0382812){1}
\usefont{T1}{ptm}{m}{n}
\rput(7.115625,0.5982813){2}
\usefont{T1}{ptm}{m}{n}
\rput(7.13375,-0.6017188){3}
\usefont{T1}{ptm}{m}{n}
\rput(6.3203125,-1.0417187){4}
\usefont{T1}{ptm}{m}{n}
\rput(5.4975,-0.64171875){5}
\usefont{T1}{ptm}{m}{n}
\rput(5.4690623,0.55828124){6}
\usefont{T1}{ptm}{m}{n}
\rput(6.7715626,0.8967187){\tiny \color{color204}1'}
\psdots[dotsize=0.12,linecolor=color204](6.689375,0.6717187)
\psdots[dotsize=0.12,linecolor=color204](7.109375,0.0117187)
\psdots[dotsize=0.12,linecolor=color204](6.729375,-0.6682813)
\psdots[dotsize=0.12,linecolor=color204](5.949375,-0.6882813)
\psdots[dotsize=0.12,linecolor=color204](5.549375,0.0117187)
\psdots[dotsize=0.12,linecolor=color204](5.969375,0.6717187)
\usefont{T1}{ptm}{m}{n}
\rput(7.3021874,-0.0032813){\tiny \color{color204}2'}
\usefont{T1}{ptm}{m}{n}
\rput(6.82,-0.9432813){\tiny \color{color204}3'}
\usefont{T1}{ptm}{m}{n}
\rput(5.8228126,-0.9432813){\tiny \color{color204}4'}
\usefont{T1}{ptm}{m}{n}
\rput(5.2992187,-0.0232813){\tiny \color{color204}5'}
\usefont{T1}{ptm}{m}{n}
\rput(5.819844,0.8767187){\tiny \color{color204}6'}
\psline[linewidth=0.04cm](1.089375,0.7717187)(1.689375,0.4117187)
\psline[linewidth=0.04cm](1.689375,0.4117187)(1.729375,-0.4282813)
\psline[linewidth=0.04cm](1.729375,-0.4282813)(1.089375,-0.7882813)
\psline[linewidth=0.04cm](1.089375,-0.7882813)(0.489375,-0.4482813)
\psline[linewidth=0.04cm](0.489375,-0.4682813)(0.469375,0.3917187)
\psline[linewidth=0.04cm](0.469375,0.3917187)(1.089375,0.7317187)
\psline[linewidth=0.04cm](3.709375,0.7717187)(4.329375,0.4117187)
\psline[linewidth=0.04cm](4.369375,-0.3882813)(3.709375,-0.7682813)
\psline[linewidth=0.04cm](3.069375,-0.4282813)(3.089375,0.4317187)
\psline[linewidth=0.04cm](5.689375,0.4117187)(6.329375,0.7517187)
\psline[linewidth=0.04cm](6.949375,0.4317187)(6.969375,-0.4482813)
\psline[linewidth=0.04cm](6.329375,-0.7682813)(5.729375,-0.4082813)
\psline[linewidth=0.04cm,linecolor=color204](3.349375,0.6517187)(4.449375,0.0117187)
\psline[linewidth=0.04cm,linecolor=color204](4.449375,0.0117187)(3.309375,-0.6882813)
\psline[linewidth=0.04cm,linecolor=color204](3.309375,-0.6882813)(3.349375,0.6717187)
\psline[linewidth=0.04cm,linecolor=color204](6.689375,0.6717187)(6.709375,-0.6482813)
\psline[linewidth=0.04cm,linecolor=color204](6.709375,-0.6482813)(5.569375,0.0117187)
\psline[linewidth=0.04cm,linecolor=color204](5.569375,0.0117187)(6.649375,0.6317187)
\end{pspicture} 
}

%% file: picture_ncp.tex
\scalebox{1} 
{
\begin{pspicture}(0,-8.181719)(9.444375,8.181719)
\pscircle[linewidth=0.04,dimen=outer](0.924375,6.9982815){0.8}
\psdots[dotsize=0.12](0.924375,7.798281)
\psdots[dotsize=0.12](1.584375,7.418281)
\psdots[dotsize=0.12](0.924375,6.1982813)
\psdots[dotsize=0.12](1.604375,6.6182814)
\psdots[dotsize=0.12](0.244375,6.6182814)
\psdots[dotsize=0.12](0.244375,7.398281)
\usefont{T1}{ptm}{m}{n}
\rput(0.87125,8.008282){1}
\usefont{T1}{ptm}{m}{n}
\rput(1.7429688,7.568281){2}
\usefont{T1}{ptm}{m}{n}
\rput(1.7720313,6.548281){3}
\usefont{T1}{ptm}{m}{n}
\rput(0.9253125,5.9682813){4}
\usefont{T1}{ptm}{m}{n}
\rput(0.05390625,6.4482813){5}
\usefont{T1}{ptm}{m}{n}
\rput(0.0596875,7.508281){6}
\pscircle[linewidth=0.04,dimen=outer](0.924375,4.1982813){0.8}
\psdots[dotsize=0.12](0.924375,4.9982815)
\psdots[dotsize=0.12](1.584375,4.6182814)
\psdots[dotsize=0.12](0.924375,3.3982813)
\psdots[dotsize=0.12](1.604375,3.8182812)
\psdots[dotsize=0.12](0.244375,3.8182812)
\psdots[dotsize=0.12](0.244375,4.5982814)
\usefont{T1}{ptm}{m}{n}
\rput(0.87125,5.208281){1}
\usefont{T1}{ptm}{m}{n}
\rput(1.7429688,4.7682815){2}
\usefont{T1}{ptm}{m}{n}
\rput(1.7720313,3.7482812){3}
\usefont{T1}{ptm}{m}{n}
\rput(0.9253125,3.1682813){4}
\usefont{T1}{ptm}{m}{n}
\rput(0.05390625,3.6482813){5}
\usefont{T1}{ptm}{m}{n}
\rput(0.0596875,4.708281){6}
\pspolygon[linewidth=0.04](0.244375,7.418281)(0.264375,6.6182814)(0.904375,6.258281)(1.584375,6.6382813)(1.564375,7.418281)(0.904375,7.798281)
\pscircle[linewidth=0.04,dimen=outer](0.924375,1.3982812){0.8}
\psdots[dotsize=0.12](0.924375,2.1982813)
\psdots[dotsize=0.12](1.584375,1.8182813)
\psdots[dotsize=0.12](0.924375,0.59828126)
\psdots[dotsize=0.12](1.604375,1.0182812)
\psdots[dotsize=0.12](0.244375,1.0182812)
\psdots[dotsize=0.12](0.244375,1.7982812)
\usefont{T1}{ptm}{m}{n}
\rput(0.87125,2.4082813){1}
\usefont{T1}{ptm}{m}{n}
\rput(1.7429688,1.9682813){2}
\usefont{T1}{ptm}{m}{n}
\rput(1.7720313,0.9482812){3}
\usefont{T1}{ptm}{m}{n}
\rput(0.9253125,0.36828125){4}
\usefont{T1}{ptm}{m}{n}
\rput(0.05390625,0.84828126){5}
\usefont{T1}{ptm}{m}{n}
\rput(0.0596875,1.9082812){6}
\pspolygon[linewidth=0.04](0.904375,2.2182813)(1.624375,0.99828124)(0.244375,1.0182812)(0.244375,1.0182812)(0.244375,1.0182812)
\psline[linewidth=0.04cm](4.524375,6.1982813)(6.924375,7.798281)
\psline[linewidth=0.04cm](4.524375,6.9982815)(5.724375,6.1982813)
\psline[linewidth=0.04cm](5.724375,6.1982813)(8.124375,7.798281)
\psline[linewidth=0.04cm](4.524375,6.9982815)(5.724375,7.798281)
\psline[linewidth=0.04cm](5.724375,7.798281)(8.124375,6.1982813)
\psline[linewidth=0.04cm](4.524375,7.798281)(6.924375,6.1982813)
\psline[linewidth=0.04cm](6.924375,7.798281)(8.124375,6.9982815)
\psline[linewidth=0.04cm](6.924375,6.1982813)(8.124375,6.9982815)
\psline[linewidth=0.04cm](8.124375,6.9982815)(8.724375,7.398281)
\psline[linewidth=0.04cm](8.124375,7.798281)(8.724375,7.398281)
\psline[linewidth=0.04cm](8.724375,7.398281)(9.324375,7.798281)
\psline[linewidth=0.04cm](8.124375,6.1982813)(9.324375,6.9982815)
\psline[linewidth=0.04cm](8.124375,6.9982815)(9.324375,6.1982813)
\psline[linewidth=0.04cm](8.724375,7.398281)(9.324375,6.9982815)
\psline[linewidth=0.04cm,linestyle=dashed,dash=0.16cm 0.16cm](4.524375,5.9982815)(4.524375,7.9982815)
\psline[linewidth=0.04cm,linestyle=dashed,dash=0.16cm 0.16cm](9.324375,5.9982815)(9.324375,7.9982815)
\psline[linewidth=0.04cm](4.524375,3.3982813)(6.924375,4.9982815)
\psline[linewidth=0.04cm](4.524375,4.1982813)(5.724375,3.3982813)
\psline[linewidth=0.04cm](5.724375,3.3982813)(8.124375,4.9982815)
\psline[linewidth=0.04cm](4.524375,4.1982813)(5.724375,4.9982815)
\psline[linewidth=0.04cm](5.724375,4.9982815)(8.124375,3.3982813)
\psline[linewidth=0.04cm](4.524375,4.9982815)(6.924375,3.3982813)
\psline[linewidth=0.04cm](6.924375,4.9982815)(8.124375,4.1982813)
\psline[linewidth=0.04cm](6.924375,3.3982813)(8.124375,4.1982813)
\psline[linewidth=0.04cm](8.124375,4.1982813)(8.724375,4.5982814)
\psline[linewidth=0.04cm](8.124375,4.9982815)(8.724375,4.5982814)
\psline[linewidth=0.04cm](8.724375,4.5982814)(9.324375,4.9982815)
\psline[linewidth=0.04cm](8.124375,3.3982813)(9.324375,4.1982813)
\psline[linewidth=0.04cm](8.124375,4.1982813)(9.324375,3.3982813)
\psline[linewidth=0.04cm](8.724375,4.5982814)(9.324375,4.1982813)
\psline[linewidth=0.04cm,linestyle=dashed,dash=0.16cm 0.16cm](4.524375,3.1982813)(4.524375,5.1982813)
\psline[linewidth=0.04cm,linestyle=dashed,dash=0.16cm 0.16cm](9.324375,3.1982813)(9.324375,5.1982813)
\psline[linewidth=0.04cm](4.524375,0.59828126)(6.924375,2.1982813)
\psline[linewidth=0.04cm](4.524375,1.3982812)(5.724375,0.59828126)
\psline[linewidth=0.04cm](5.724375,0.59828126)(8.124375,2.1982813)
\psline[linewidth=0.04cm](4.524375,1.3982812)(5.724375,2.1982813)
\psline[linewidth=0.04cm](5.724375,2.1982813)(8.124375,0.59828126)
\psline[linewidth=0.04cm](4.524375,2.1982813)(6.924375,0.59828126)
\psline[linewidth=0.04cm](6.924375,2.1982813)(8.124375,1.3982812)
\psline[linewidth=0.04cm](6.924375,0.59828126)(8.124375,1.3982812)
\psline[linewidth=0.04cm](8.124375,1.3982812)(8.724375,1.7982812)
\psline[linewidth=0.04cm](8.124375,2.1982813)(8.724375,1.7982812)
\psline[linewidth=0.04cm](8.724375,1.7982812)(9.324375,2.1982813)
\psline[linewidth=0.04cm](8.124375,0.59828126)(9.324375,1.3982812)
\psline[linewidth=0.04cm](8.124375,1.3982812)(9.324375,0.59828126)
\psline[linewidth=0.04cm](8.724375,1.7982812)(9.324375,1.3982812)
\psline[linewidth=0.04cm,linestyle=dashed,dash=0.16cm 0.16cm](4.524375,0.39828125)(4.524375,2.3982813)
\psline[linewidth=0.04cm,linestyle=dashed,dash=0.16cm 0.16cm](9.324375,0.39828125)(9.324375,2.3982813)
\pscircle[linewidth=0.04,dimen=outer](0.924375,-1.4017187){0.8}
\psdots[dotsize=0.12](0.924375,-0.6017187)
\psdots[dotsize=0.12](1.584375,-0.9817188)
\psdots[dotsize=0.12](0.924375,-2.2017188)
\psdots[dotsize=0.12](1.604375,-1.7817187)
\psdots[dotsize=0.12](0.244375,-1.7817187)
\psdots[dotsize=0.12](0.244375,-1.0017188)
\usefont{T1}{ptm}{m}{n}
\rput(0.87125,-0.39171875){1}
\usefont{T1}{ptm}{m}{n}
\rput(1.7429688,-0.83171874){2}
\usefont{T1}{ptm}{m}{n}
\rput(1.7720313,-1.8517188){3}
\usefont{T1}{ptm}{m}{n}
\rput(0.9253125,-2.4317188){4}
\usefont{T1}{ptm}{m}{n}
\rput(0.05390625,-1.9517188){5}
\usefont{T1}{ptm}{m}{n}
\rput(0.0596875,-0.89171875){6}
\pscircle[linewidth=0.04,dimen=outer](0.924375,-4.201719){0.8}
\psdots[dotsize=0.12](0.924375,-3.4017189)
\psdots[dotsize=0.12](1.584375,-3.7817187)
\psdots[dotsize=0.12](0.924375,-5.0017185)
\psdots[dotsize=0.12](1.604375,-4.581719)
\psdots[dotsize=0.12](0.244375,-4.581719)
\psdots[dotsize=0.12](0.244375,-3.8017187)
\usefont{T1}{ptm}{m}{n}
\rput(0.87125,-3.1917188){1}
\usefont{T1}{ptm}{m}{n}
\rput(1.7429688,-3.6317186){2}
\usefont{T1}{ptm}{m}{n}
\rput(1.7720313,-4.6517186){3}
\usefont{T1}{ptm}{m}{n}
\rput(0.9253125,-5.2317185){4}
\usefont{T1}{ptm}{m}{n}
\rput(0.05390625,-4.7517185){5}
\usefont{T1}{ptm}{m}{n}
\rput(0.0596875,-3.6917188){6}
\pscircle[linewidth=0.04,dimen=outer](0.924375,-7.0017185){0.8}
\psdots[dotsize=0.12](0.924375,-6.201719)
\psdots[dotsize=0.12](1.584375,-6.581719)
\psdots[dotsize=0.12](0.924375,-7.8017187)
\psdots[dotsize=0.12](1.604375,-7.3817186)
\psdots[dotsize=0.12](0.244375,-7.3817186)
\psdots[dotsize=0.12](0.244375,-6.601719)
\usefont{T1}{ptm}{m}{n}
\rput(0.87125,-5.991719){1}
\usefont{T1}{ptm}{m}{n}
\rput(1.7429688,-6.431719){2}
\usefont{T1}{ptm}{m}{n}
\rput(1.7720313,-7.451719){3}
\usefont{T1}{ptm}{m}{n}
\rput(0.9253125,-8.031719){4}
\usefont{T1}{ptm}{m}{n}
\rput(0.05390625,-7.5517187){5}
\usefont{T1}{ptm}{m}{n}
\rput(0.0596875,-6.491719){6}
\psline[linewidth=0.04cm](4.524375,-2.2017188)(6.924375,-0.6017187)
\psline[linewidth=0.04cm](4.524375,-1.4017187)(5.724375,-2.2017188)
\psline[linewidth=0.04cm](5.724375,-2.2017188)(8.124375,-0.6017187)
\psline[linewidth=0.04cm](4.524375,-1.4017187)(5.724375,-0.6017187)
\psline[linewidth=0.04cm](5.724375,-0.6017187)(8.124375,-2.2017188)
\psline[linewidth=0.04cm](4.524375,-0.6017187)(6.924375,-2.2017188)
\psline[linewidth=0.04cm](6.924375,-0.6017187)(8.124375,-1.4017187)
\psline[linewidth=0.04cm](6.924375,-2.2017188)(8.124375,-1.4017187)
\psline[linewidth=0.04cm](8.124375,-1.4017187)(8.724375,-1.0017188)
\psline[linewidth=0.04cm](8.124375,-0.6017187)(8.724375,-1.0017188)
\psline[linewidth=0.04cm](8.724375,-1.0017188)(9.324375,-0.6017187)
\psline[linewidth=0.04cm](8.124375,-2.2017188)(9.324375,-1.4017187)
\psline[linewidth=0.04cm](8.124375,-1.4017187)(9.324375,-2.2017188)
\psline[linewidth=0.04cm](8.724375,-1.0017188)(9.324375,-1.4017187)
\psline[linewidth=0.04cm,linestyle=dashed,dash=0.16cm 0.16cm](4.524375,-2.4017189)(4.524375,-0.40171874)
\psline[linewidth=0.04cm,linestyle=dashed,dash=0.16cm 0.16cm](9.324375,-2.4017189)(9.324375,-0.40171874)
\psline[linewidth=0.04cm](4.524375,-5.0017185)(6.924375,-3.4017189)
\psline[linewidth=0.04cm](4.524375,-4.201719)(5.724375,-5.0017185)
\psline[linewidth=0.04cm](5.724375,-5.0017185)(8.124375,-3.4017189)
\psline[linewidth=0.04cm](4.524375,-4.201719)(5.724375,-3.4017189)
\psline[linewidth=0.04cm](5.724375,-3.4017189)(8.124375,-5.0017185)
\psline[linewidth=0.04cm](4.524375,-3.4017189)(6.924375,-5.0017185)
\psline[linewidth=0.04cm](6.924375,-3.4017189)(8.124375,-4.201719)
\psline[linewidth=0.04cm](6.924375,-5.0017185)(8.124375,-4.201719)
\psline[linewidth=0.04cm](8.124375,-4.201719)(8.724375,-3.8017187)
\psline[linewidth=0.04cm](8.124375,-3.4017189)(8.724375,-3.8017187)
\psline[linewidth=0.04cm](8.724375,-3.8017187)(9.324375,-3.4017189)
\psline[linewidth=0.04cm](8.124375,-5.0017185)(9.324375,-4.201719)
\psline[linewidth=0.04cm](8.124375,-4.201719)(9.324375,-5.0017185)
\psline[linewidth=0.04cm](8.724375,-3.8017187)(9.324375,-4.201719)
\psline[linewidth=0.04cm,linestyle=dashed,dash=0.16cm 0.16cm](4.524375,-5.201719)(4.524375,-3.2017188)
\psline[linewidth=0.04cm,linestyle=dashed,dash=0.16cm 0.16cm](9.324375,-5.201719)(9.324375,-3.2017188)
\psline[linewidth=0.04cm](4.524375,-7.601719)(6.924375,-6.0017185)
\psline[linewidth=0.04cm](4.524375,-6.8017187)(5.724375,-7.601719)
\psline[linewidth=0.04cm](5.724375,-7.601719)(8.124375,-6.0017185)
\psline[linewidth=0.04cm](4.524375,-6.8017187)(5.724375,-6.0017185)
\psline[linewidth=0.04cm](5.724375,-6.0017185)(8.124375,-7.601719)
\psline[linewidth=0.04cm](4.524375,-6.0017185)(6.924375,-7.601719)
\psline[linewidth=0.04cm](6.924375,-6.0017185)(8.124375,-6.8017187)
\psline[linewidth=0.04cm](6.924375,-7.601719)(8.124375,-6.8017187)
\psline[linewidth=0.04cm](8.124375,-6.8017187)(8.724375,-6.4017186)
\psline[linewidth=0.04cm](8.124375,-6.0017185)(8.724375,-6.4017186)
\psline[linewidth=0.04cm](8.724375,-6.4017186)(9.324375,-6.0017185)
\psline[linewidth=0.04cm](8.124375,-7.601719)(9.324375,-6.8017187)
\psline[linewidth=0.04cm](8.124375,-6.8017187)(9.324375,-7.601719)
\psline[linewidth=0.04cm](8.724375,-6.4017186)(9.324375,-6.8017187)
\psline[linewidth=0.04cm,linestyle=dashed,dash=0.16cm 0.16cm](4.524375,-7.8017187)(4.524375,-5.8017187)
\psline[linewidth=0.04cm,linestyle=dashed,dash=0.16cm 0.16cm](9.324375,-7.8017187)(9.324375,-5.8017187)
\psdots[dotsize=0.2](5.124375,1.7982812)
\psdots[dotsize=0.2](6.324375,0.99828124)
\psdots[dotsize=0.2](7.524375,1.7982812)
\psdots[dotsize=0.2](8.724375,0.99828124)
\psdots[dotsize=0.2](4.524375,7.798281)
\psdots[dotsize=0.2](4.524375,6.9982815)
\psdots[dotsize=0.2](4.524375,6.1982813)
\psdots[dotsize=0.2](5.124375,7.398281)
\psdots[dotsize=0.2](5.124375,6.5982814)
\psdots[dotsize=0.2](5.724375,6.1982813)
\psdots[dotsize=0.2](5.724375,6.9982815)
\psdots[dotsize=0.2](5.724375,7.798281)
\psdots[dotsize=0.2](6.324375,7.398281)
\psdots[dotsize=0.2](6.324375,6.5982814)
\psdots[dotsize=0.2](6.924375,6.1982813)
\psdots[dotsize=0.2](6.924375,6.9982815)
\psdots[dotsize=0.2](6.924375,7.798281)
\psdots[dotsize=0.2](7.524375,7.398281)
\psdots[dotsize=0.2](7.524375,6.5982814)
\psdots[dotsize=0.2](8.124375,6.1982813)
\psdots[dotsize=0.2](8.124375,6.9982815)
\psdots[dotsize=0.2](8.124375,7.798281)
\psdots[dotsize=0.2](8.724375,7.398281)
\psdots[dotsize=0.2](9.324375,7.798281)
\psdots[dotsize=0.2](9.324375,6.9982815)
\psdots[dotsize=0.2](8.724375,6.5982814)
\psdots[dotsize=0.2](9.324375,6.1982813)
\pspolygon[linewidth=0.04](0.924375,-2.2017188)(0.224375,-0.9817188)(1.584375,-0.9817188)
\psdots[dotsize=0.2](6.324375,-1.0017188)
\psdots[dotsize=0.2](8.724375,-1.0017188)
\psdots[dotsize=0.2](7.524375,-1.8017187)
\psdots[dotsize=0.2](5.124375,-1.8017187)
\psline[linewidth=0.04cm](0.904375,-6.201719)(1.564375,-6.601719)
\psline[linewidth=0.04cm](1.584375,-7.3617187)(0.904375,-7.7817187)
\psline[linewidth=0.04cm](0.244375,-7.3617187)(0.264375,-6.601719)
\psdots[dotsize=0.2](9.324375,-7.601719)
\psdots[dotsize=0.2](6.924375,-7.601719)
\psdots[dotsize=0.2](4.524375,-7.601719)
\psdots[dotsize=0.2](8.124375,-6.0017185)
\psdots[dotsize=0.2](5.724375,-6.0017185)
\psdots[dotsize=0.2](8.124375,-5.0017185)
\psdots[dotsize=0.2](5.724375,-5.0017185)
\psdots[dotsize=0.2](6.924375,-3.4017189)
\psdots[dotsize=0.2](9.324375,-3.4017189)
\psdots[dotsize=0.2](4.524375,-3.4017189)
\psline[linewidth=0.04cm](1.564375,-3.7817187)(1.584375,-4.581719)
\psline[linewidth=0.04cm](0.904375,-4.9417186)(0.264375,-4.541719)
\psline[linewidth=0.04cm](0.264375,-3.8017187)(0.924375,-3.4417188)
\end{pspicture} 
}

%% file: picture_ncp_Dn_1.tex
\scalebox{1} 
{
\begin{pspicture}(0,-1.2317187)(8.335313,1.2317187)
\pscircle[linewidth=0.04,dimen=outer](1.2732812,0.0248429){0.8}
\psdots[dotsize=0.12](1.23,0.8282811)
\psdots[dotsize=0.12](1.23,-0.7717188)
\psdots[dotsize=0.12](1.23,0.0282812)
\psdots[dotsize=0.12](0.49,0.0482812)
\psdots[dotsize=0.12](2.07,0.0482812)
\psdots[dotsize=0.12](1.81,0.6082814)
\psdots[dotsize=0.12](1.8299999,-0.5517187)
\psdots[dotsize=0.12](0.66999996,-0.5117188)
\psdots[dotsize=0.12](0.71,0.5682812)
\usefont{T1}{ptm}{m}{n}
\rput(1.090625,1.0582812){1}
\usefont{T1}{ptm}{m}{n}
\rput(1.9057811,0.7982812){2}
\usefont{T1}{ptm}{m}{n}
\rput(2.2329688,0.0182812){3}
\usefont{T1}{ptm}{m}{n}
\rput(1.9328126,-0.7217188){4}
\usefont{T1}{ptm}{m}{n}
\rput(1.1518749,-1.0417187){-1}
\usefont{T1}{ptm}{m}{n}
\rput(0.35546875,-0.7017187){-2}
\usefont{T1}{ptm}{m}{n}
\rput(0.11296875,0.0382812){-3}
\usefont{T1}{ptm}{m}{n}
\rput(0.39921874,0.7182813){-4}
\pscircle[linewidth=0.04,dimen=outer](4.273281,0.0048429){0.8}
\psdots[dotsize=0.12](4.23,0.8082812)
\psdots[dotsize=0.12](4.23,-0.7917187)
\psdots[dotsize=0.12](4.23,0.0082812)
\psdots[dotsize=0.12](3.4900002,0.0282812)
\psdots[dotsize=0.12](5.07,0.0282812)
\psdots[dotsize=0.12](4.81,0.5882811)
\psdots[dotsize=0.12](4.83,-0.5717187)
\psdots[dotsize=0.12](3.67,-0.5317187)
\psdots[dotsize=0.12](3.71,0.5482812)
\usefont{T1}{ptm}{m}{n}
\rput(4.0906253,1.0382812){1}
\usefont{T1}{ptm}{m}{n}
\rput(4.9057813,0.7782812){2}
\usefont{T1}{ptm}{m}{n}
\rput(5.232969,-0.0017188){3}
\usefont{T1}{ptm}{m}{n}
\rput(4.932812,-0.7417188){4}
\usefont{T1}{ptm}{m}{n}
\rput(4.151875,-1.0617187){-1}
\usefont{T1}{ptm}{m}{n}
\rput(3.3554688,-0.7217188){-2}
\usefont{T1}{ptm}{m}{n}
\rput(3.1129687,0.0182812){-3}
\usefont{T1}{ptm}{m}{n}
\rput(3.3992188,0.6982813){-4}
\psline[linewidth=0.04](0.68999994,0.5882811)(1.7900001,0.5882811)(1.81,-0.4917188)(0.68999994,-0.4917188)(0.71,0.6082814)(0.71,0.5882811)
\psline[linewidth=0.04cm](4.21,0.7682812)(4.79,0.5282812)
\psline[linewidth=0.04cm](5.0099998,0.0482812)(4.75,-0.5517187)
\psline[linewidth=0.04cm](4.23,-0.7317188)(3.71,-0.4717188)
\psline[linewidth=0.04cm](3.55,0.0482812)(3.7500002,0.5282812)
\pscircle[linewidth=0.04,dimen=outer](7.253281,-0.0151571){0.8}
\psdots[dotsize=0.12](7.21,0.7882812)
\psdots[dotsize=0.12](7.21,-0.8117187)
\psdots[dotsize=0.12](7.21,-0.0117188)
\psdots[dotsize=0.12](6.47,0.0082812)
\psdots[dotsize=0.12](8.050001,0.0082812)
\psdots[dotsize=0.12](7.79,0.5682812)
\psdots[dotsize=0.12](7.81,-0.5917187)
\psdots[dotsize=0.12](6.65,-0.5517187)
\psdots[dotsize=0.12](6.69,0.5282812)
\usefont{T1}{ptm}{m}{n}
\rput(7.070625,1.0182812){1}
\usefont{T1}{ptm}{m}{n}
\rput(7.8857813,0.7582812){2}
\usefont{T1}{ptm}{m}{n}
\rput(8.212969,-0.0217188){3}
\usefont{T1}{ptm}{m}{n}
\rput(7.912813,-0.7617188){4}
\usefont{T1}{ptm}{m}{n}
\rput(7.131875,-1.0817187){-1}
\usefont{T1}{ptm}{m}{n}
\rput(6.335469,-0.7417188){-2}
\usefont{T1}{ptm}{m}{n}
\rput(6.092969,-0.0017188){-3}
\usefont{T1}{ptm}{m}{n}
\rput(6.3792186,0.6782813){-4}
\psline[linewidth=0.04cm](7.19,0.7482812)(7.769999,0.5082812)
\psline[linewidth=0.04cm](7.9900002,0.0282812)(7.729999,-0.5717187)
\psline[linewidth=0.04cm](7.21,-0.7517188)(6.69,-0.4917188)
\psline[linewidth=0.04cm](6.5299997,0.0282812)(6.73,0.5082812)
\psline[linewidth=0.04cm](7.21,0.0282812)(7.21,0.7682812)
\psline[linewidth=0.04cm](7.19,0.0082812)(7.769999,0.5282812)
\psline[linewidth=0.04cm](6.63,-0.5117188)(7.17,-0.0317187)
\psline[linewidth=0.04cm](7.19,-0.0117188)(7.21,-0.7917187)
\usefont{T1}{ptm}{m}{n}
\rput(6.9940624,-0.4417188){+}
\usefont{T1}{ptm}{m}{n}
\rput(7.3264065,0.4182813){-}
\end{pspicture} 
}

%% file: picture_ncp_Dn_2.tex
\scalebox{1} 
{
\begin{pspicture}(0,-0.615)(11.815,0.615)
\psline[linewidth=0.03cm](0.0,-0.4)(0.6,0.2)
\psline[linewidth=0.03cm](0.4,-0.4)(1.0,0.2)
\psline[linewidth=0.03cm](0.0,0.0)(0.4,-0.4)
\psline[linewidth=0.03cm](0.0,0.0)(0.2,0.2)
\psline[linewidth=0.03cm](0.2,0.2)(0.8,-0.4)
\psline[linewidth=0.03cm](0.0,0.0)(0.4,0.0)
\psline[linewidth=0.03cm](0.6,0.2)(1.2,-0.4)
\psline[linewidth=0.03cm](0.8,-0.4)(1.4,0.2)
\psline[linewidth=0.03cm](1.0,0.2)(1.6,-0.4)
\psline[linewidth=0.03cm](1.2,-0.4)(1.8,0.2)
\psline[linewidth=0.03cm](1.4,0.2)(2.0,-0.4)
\psline[linewidth=0.03cm](1.6,-0.4)(2.2,0.2)
\psline[linewidth=0.03cm](1.8,0.2)(2.4,-0.4)
\psline[linewidth=0.03cm](2.0,-0.4)(2.6,0.2)
\psline[linewidth=0.03cm](2.2,0.2)(2.8,-0.4)
\psline[linewidth=0.03cm](2.4,-0.4)(3.0,0.2)
\psline[linewidth=0.03cm](2.6,0.2)(3.2,-0.4)
\psline[linewidth=0.03cm](2.8,-0.4)(3.4,0.2)
\psline[linewidth=0.03cm](3.0,0.2)(3.6,-0.4)
\psline[linewidth=0.03cm](3.2,-0.4)(3.8,0.2)
\psline[linewidth=0.03cm](3.4,0.2)(4.0,-0.4)
\psline[linewidth=0.03cm](3.6,-0.4)(4.2,0.2)
\psline[linewidth=0.03cm](3.8,0.2)(4.4,-0.4)
\psline[linewidth=0.03cm](4.0,-0.4)(4.6,0.2)
\psline[linewidth=0.03cm](4.2,0.2)(4.8,-0.4)
\psline[linewidth=0.03cm](4.4,-0.4)(5.0,0.2)
\psline[linewidth=0.03cm](4.6,0.2)(5.2,-0.4)
\psline[linewidth=0.03cm](4.8,-0.4)(5.4,0.2)
\psline[linewidth=0.03cm](5.0,0.2)(5.6,-0.4)
\psline[linewidth=0.03cm,linestyle=dashed,dash=0.16cm 0.16cm](0.0,-0.6)(0.0,0.6)
\psline[linewidth=0.03cm,linestyle=dashed,dash=0.16cm 0.16cm](5.6,-0.6)(5.6,0.6)
\psline[linewidth=0.03cm](5.4,0.2)(5.6,0.0)
\psline[linewidth=0.03cm](5.6,0.0)(5.2,-0.4)
\psline[linewidth=0.03cm](5.6,0.0)(0.4,0.0)
\psdots[dotsize=0.12](0.185,0.215)
\psdots[dotsize=0.12](0.185,0.015)
\psdots[dotsize=0.12](0.585,-0.185)
\psdots[dotsize=0.12](0.985,0.215)
\psdots[dotsize=0.12](0.985,0.015)
\psdots[dotsize=0.12](1.385,-0.185)
\psdots[dotsize=0.12](1.785,0.215)
\psdots[dotsize=0.12](1.785,0.015)
\psdots[dotsize=0.12](2.185,-0.185)
\psdots[dotsize=0.12](2.585,0.215)
\psdots[dotsize=0.12](2.585,0.015)
\psdots[dotsize=0.12](2.985,-0.185)
\psdots[dotsize=0.12](3.385,0.215)
\psdots[dotsize=0.12](3.385,0.015)
\psdots[dotsize=0.12](3.785,-0.185)
\psdots[dotsize=0.12](4.185,0.215)
\psdots[dotsize=0.12](4.185,0.015)
\psdots[dotsize=0.12](4.585,-0.185)
\psdots[dotsize=0.12](4.985,0.215)
\psdots[dotsize=0.12](4.985,0.015)
\psdots[dotsize=0.12](5.385,-0.185)
\psline[linewidth=0.03cm](6.2,-0.4)(6.8,0.2)
\psline[linewidth=0.03cm](6.6,-0.4)(7.2,0.2)
\psline[linewidth=0.03cm](6.2,0.0)(6.6,-0.4)
\psline[linewidth=0.03cm](6.2,0.0)(6.4,0.2)
\psline[linewidth=0.03cm](6.4,0.2)(7.0,-0.4)
\psline[linewidth=0.03cm](6.2,0.0)(6.6,0.0)
\psline[linewidth=0.03cm](6.8,0.2)(7.4,-0.4)
\psline[linewidth=0.03cm](7.0,-0.4)(7.6,0.2)
\psline[linewidth=0.03cm](7.2,0.2)(7.8,-0.4)
\psline[linewidth=0.03cm](7.4,-0.4)(8.0,0.2)
\psline[linewidth=0.03cm](7.6,0.2)(8.2,-0.4)
\psline[linewidth=0.03cm](7.8,-0.4)(8.4,0.2)
\psline[linewidth=0.03cm](8.0,0.2)(8.6,-0.4)
\psline[linewidth=0.03cm](8.2,-0.4)(8.8,0.2)
\psline[linewidth=0.03cm](8.4,0.2)(9.0,-0.4)
\psline[linewidth=0.03cm](8.6,-0.4)(9.2,0.2)
\psline[linewidth=0.03cm](8.8,0.2)(9.4,-0.4)
\psline[linewidth=0.03cm](9.0,-0.4)(9.6,0.2)
\psline[linewidth=0.03cm](9.2,0.2)(9.8,-0.4)
\psline[linewidth=0.03cm](9.4,-0.4)(10.0,0.2)
\psline[linewidth=0.03cm](9.6,0.2)(10.2,-0.4)
\psline[linewidth=0.03cm](9.8,-0.4)(10.4,0.2)
\psline[linewidth=0.03cm](10.0,0.2)(10.6,-0.4)
\psline[linewidth=0.03cm](10.2,-0.4)(10.8,0.2)
\psline[linewidth=0.03cm](10.4,0.2)(11.0,-0.4)
\psline[linewidth=0.03cm](10.6,-0.4)(11.2,0.2)
\psline[linewidth=0.03cm](10.8,0.2)(11.4,-0.4)
\psline[linewidth=0.03cm](11.0,-0.4)(11.6,0.2)
\psline[linewidth=0.03cm](11.2,0.2)(11.8,-0.4)
\psline[linewidth=0.03cm,linestyle=dashed,dash=0.16cm 0.16cm](6.2,-0.6)(6.2,0.6)
\psline[linewidth=0.03cm,linestyle=dashed,dash=0.16cm 0.16cm](11.8,-0.6)(11.8,0.6)
\psline[linewidth=0.03cm](11.6,0.2)(11.8,0.0)
\psline[linewidth=0.03cm](11.8,0.0)(11.4,-0.4)
\psline[linewidth=0.03cm](11.8,0.0)(6.6,0.0)
\psdots[dotsize=0.12](6.585,-0.385)
\psdots[dotsize=0.12](7.385,-0.385)
\psdots[dotsize=0.12](8.185,-0.385)
\psdots[dotsize=0.12](8.985,-0.385)
\psdots[dotsize=0.12](9.785,-0.385)
\psdots[dotsize=0.12](10.585,-0.385)
\psdots[dotsize=0.12](11.385,-0.385)
\end{pspicture} 
}

%% file: picture_ncp_Dn_3.tex
\scalebox{1} 
{
\begin{pspicture}(0,-2.5317187)(8.335313,2.5317187)
\pscircle[linewidth=0.04,dimen=outer](1.2732812,1.3248429){0.8}
\psdots[dotsize=0.12](1.23,2.128281)
\psdots[dotsize=0.12](1.23,0.5282812)
\psdots[dotsize=0.12](1.23,1.3282812)
\psdots[dotsize=0.12](0.49,1.3482811)
\psdots[dotsize=0.12](2.07,1.3482811)
\psdots[dotsize=0.12](1.81,1.9082814)
\psdots[dotsize=0.12](1.8299999,0.7482813)
\psdots[dotsize=0.12](0.66999996,0.7882812)
\psdots[dotsize=0.12](0.71,1.8682812)
\usefont{T1}{ptm}{m}{n}
\rput(1.090625,2.3582811){1}
\usefont{T1}{ptm}{m}{n}
\rput(1.9057811,2.0982811){2}
\usefont{T1}{ptm}{m}{n}
\rput(2.2329688,1.3182812){3}
\usefont{T1}{ptm}{m}{n}
\rput(1.9328126,0.5782812){4}
\usefont{T1}{ptm}{m}{n}
\rput(1.1518749,0.2582813){-1}
\usefont{T1}{ptm}{m}{n}
\rput(0.35546875,0.5982813){-2}
\usefont{T1}{ptm}{m}{n}
\rput(0.11296875,1.3382812){-3}
\usefont{T1}{ptm}{m}{n}
\rput(0.39921874,2.0182812){-4}
\pscircle[linewidth=0.04,dimen=outer](4.273281,1.304843){0.8}
\psdots[dotsize=0.12](4.23,2.1082811)
\psdots[dotsize=0.12](4.23,0.5082813)
\psdots[dotsize=0.12](4.23,1.3082812)
\psdots[dotsize=0.12](3.4900002,1.3282812)
\psdots[dotsize=0.12](5.07,1.3282812)
\psdots[dotsize=0.12](4.81,1.8882811)
\psdots[dotsize=0.12](4.83,0.7282813)
\psdots[dotsize=0.12](3.67,0.7682813)
\psdots[dotsize=0.12](3.71,1.8482811)
\usefont{T1}{ptm}{m}{n}
\rput(4.0906253,2.3382812){1}
\usefont{T1}{ptm}{m}{n}
\rput(4.9057813,2.0782812){2}
\usefont{T1}{ptm}{m}{n}
\rput(5.232969,1.2982812){3}
\usefont{T1}{ptm}{m}{n}
\rput(4.932812,0.5582812){4}
\usefont{T1}{ptm}{m}{n}
\rput(4.151875,0.2382813){-1}
\usefont{T1}{ptm}{m}{n}
\rput(3.3554688,0.5782812){-2}
\usefont{T1}{ptm}{m}{n}
\rput(3.1129687,1.3182812){-3}
\usefont{T1}{ptm}{m}{n}
\rput(3.3992188,1.9982812){-4}
\psline[linewidth=0.04](0.68999994,1.8882811)(1.7900001,1.8882811)(1.81,0.8082812)(0.68999994,0.8082812)(0.71,1.9082814)(0.71,1.8882811)
\pscircle[linewidth=0.04,dimen=outer](7.253281,1.2848428){0.8}
\psdots[dotsize=0.12](7.21,2.0882812)
\psdots[dotsize=0.12](7.21,0.4882813)
\psdots[dotsize=0.12](7.21,1.2882812)
\psdots[dotsize=0.12](6.47,1.3082812)
\psdots[dotsize=0.12](8.050001,1.3082812)
\psdots[dotsize=0.12](7.79,1.8682812)
\psdots[dotsize=0.12](7.81,0.7082813)
\psdots[dotsize=0.12](6.65,0.7482813)
\psdots[dotsize=0.12](6.69,1.8282812)
\usefont{T1}{ptm}{m}{n}
\rput(7.070625,2.3182812){1}
\usefont{T1}{ptm}{m}{n}
\rput(7.8857813,2.0582812){2}
\usefont{T1}{ptm}{m}{n}
\rput(8.212969,1.2782812){3}
\usefont{T1}{ptm}{m}{n}
\rput(7.912813,0.5382812){4}
\usefont{T1}{ptm}{m}{n}
\rput(7.131875,0.2182813){-1}
\usefont{T1}{ptm}{m}{n}
\rput(6.335469,0.5582812){-2}
\usefont{T1}{ptm}{m}{n}
\rput(6.092969,1.2982812){-3}
\usefont{T1}{ptm}{m}{n}
\rput(6.3792186,1.9782813){-4}
\pscircle[linewidth=0.04,dimen=outer](1.2732812,-1.2751571){0.8}
\psdots[dotsize=0.12](1.23,-0.4717189)
\psdots[dotsize=0.12](1.23,-2.0717187)
\psdots[dotsize=0.12](1.23,-1.2717187)
\psdots[dotsize=0.12](0.49,-1.2517188)
\psdots[dotsize=0.12](2.07,-1.2517188)
\psdots[dotsize=0.12](1.81,-0.6917186)
\psdots[dotsize=0.12](1.8299999,-1.8517187)
\psdots[dotsize=0.12](0.66999996,-1.8117188)
\psdots[dotsize=0.12](0.71,-0.7317188)
\usefont{T1}{ptm}{m}{n}
\rput(1.090625,-0.2417188){1}
\usefont{T1}{ptm}{m}{n}
\rput(1.9057811,-0.5017188){2}
\usefont{T1}{ptm}{m}{n}
\rput(2.2329688,-1.2817189){3}
\usefont{T1}{ptm}{m}{n}
\rput(1.9328126,-2.0217187){4}
\usefont{T1}{ptm}{m}{n}
\rput(1.1518749,-2.3417187){-1}
\usefont{T1}{ptm}{m}{n}
\rput(0.35546875,-2.0017188){-2}
\usefont{T1}{ptm}{m}{n}
\rput(0.11296875,-1.2617188){-3}
\usefont{T1}{ptm}{m}{n}
\rput(0.39921874,-0.5817187){-4}
\pscircle[linewidth=0.04,dimen=outer](4.273281,-1.2951571){0.8}
\psdots[dotsize=0.12](4.23,-0.4917188)
\psdots[dotsize=0.12](4.23,-2.0917187)
\psdots[dotsize=0.12](4.23,-1.2917188)
\psdots[dotsize=0.12](3.4900002,-1.2717187)
\psdots[dotsize=0.12](5.07,-1.2717187)
\psdots[dotsize=0.12](4.81,-0.7117189)
\psdots[dotsize=0.12](4.83,-1.8717186)
\psdots[dotsize=0.12](3.67,-1.8317187)
\psdots[dotsize=0.12](3.71,-0.7517188)
\usefont{T1}{ptm}{m}{n}
\rput(4.0906253,-0.2617188){1}
\usefont{T1}{ptm}{m}{n}
\rput(4.9057813,-0.5217188){2}
\usefont{T1}{ptm}{m}{n}
\rput(5.232969,-1.3017188){3}
\usefont{T1}{ptm}{m}{n}
\rput(4.932812,-2.0417187){4}
\usefont{T1}{ptm}{m}{n}
\rput(4.151875,-2.3617187){-1}
\usefont{T1}{ptm}{m}{n}
\rput(3.3554688,-2.0217187){-2}
\usefont{T1}{ptm}{m}{n}
\rput(3.1129687,-1.2817189){-3}
\usefont{T1}{ptm}{m}{n}
\rput(3.3992188,-0.6017187){-4}
\psline[linewidth=0.04cm](4.21,-0.5317188)(4.79,-0.7717188)
\psline[linewidth=0.04cm](5.0099998,-1.2517188)(4.75,-1.8517187)
\psline[linewidth=0.04cm](4.23,-2.0317187)(3.71,-1.7717187)
\psline[linewidth=0.04cm](3.55,-1.2517188)(3.7500002,-0.7717188)
\pscircle[linewidth=0.04,dimen=outer](7.253281,-1.315157){0.8}
\psdots[dotsize=0.12](7.21,-0.5117188)
\psdots[dotsize=0.12](7.21,-2.1117187)
\psdots[dotsize=0.12](7.21,-1.3117188)
\psdots[dotsize=0.12](6.47,-1.2917188)
\psdots[dotsize=0.12](8.050001,-1.2917188)
\psdots[dotsize=0.12](7.79,-0.7317188)
\psdots[dotsize=0.12](7.81,-1.8917187)
\psdots[dotsize=0.12](6.65,-1.8517187)
\psdots[dotsize=0.12](6.69,-0.7717188)
\usefont{T1}{ptm}{m}{n}
\rput(7.070625,-0.2817188){1}
\usefont{T1}{ptm}{m}{n}
\rput(7.8857813,-0.5417188){2}
\usefont{T1}{ptm}{m}{n}
\rput(8.212969,-1.3217188){3}
\usefont{T1}{ptm}{m}{n}
\rput(7.912813,-2.0617187){4}
\usefont{T1}{ptm}{m}{n}
\rput(7.131875,-2.3817186){-1}
\usefont{T1}{ptm}{m}{n}
\rput(6.335469,-2.0417187){-2}
\usefont{T1}{ptm}{m}{n}
\rput(6.092969,-1.3017188){-3}
\usefont{T1}{ptm}{m}{n}
\rput(6.3792186,-0.6217187){-4}
\psline[linewidth=0.04cm](4.22,2.0917187)(5.04,1.3317187)
\psline[linewidth=0.04cm](5.02,1.3317187)(4.24,0.5317187)
\psline[linewidth=0.04cm](4.22,0.5717187)(3.52,1.2917187)
\psline[linewidth=0.04cm](3.48,1.3717186)(4.2,2.0917187)
\psline[linewidth=0.04cm](0.72,-0.7682813)(1.22,-0.5082813)
\psline[linewidth=0.04cm](1.76,-0.7082813)(2.02,-1.2482812)
\psline[linewidth=0.04cm](1.8,-1.7882813)(1.22,-2.0282812)
\psline[linewidth=0.04cm](0.7,-1.7682813)(0.54,-1.2282813)
\psline[linewidth=0.04cm](7.22,-0.5282813)(7.76,-0.7682813)
\psline[linewidth=0.04cm](7.76,-0.7682813)(8.02,-1.3282813)
\psline[linewidth=0.04cm](8.02,-1.3282813)(7.78,-1.8682812)
\psline[linewidth=0.04cm](7.76,-1.8682812)(7.24,-2.0682814)
\psline[linewidth=0.04cm](7.24,-2.0682814)(6.68,-1.8282813)
\psline[linewidth=0.04cm](6.68,-1.8282813)(6.5,-1.2882813)
\psline[linewidth=0.04cm](6.5,-1.2882813)(6.72,-0.7682813)
\psline[linewidth=0.04cm](6.72,-0.7682813)(7.2,-0.5482813)
\end{pspicture} 
}

%% file: picture_ncp_Dn_4.tex
\scalebox{1} 
{
\begin{pspicture}(0,-1.165)(9.28,1.165)
\psline[linewidth=0.04cm](0.63,0.655)(1.03,1.055)
\psline[linewidth=0.04cm](0.63,0.655)(1.03,0.255)
\psline[linewidth=0.04cm](0.63,0.655)(1.43,0.655)
\psline[linewidth=0.04cm](1.03,1.055)(1.43,0.655)
\psline[linewidth=0.04cm](1.03,0.255)(1.43,0.655)
\psline[linewidth=0.04cm](1.43,0.655)(8.63,0.655)
\psline[linewidth=0.04cm](1.43,0.655)(1.83,1.055)
\psline[linewidth=0.04cm](1.83,1.055)(2.23,0.655)
\psline[linewidth=0.04cm](2.23,0.655)(2.63,1.055)
\psline[linewidth=0.04cm](2.63,1.055)(3.03,0.655)
\psline[linewidth=0.04cm](3.03,0.655)(3.43,1.055)
\psline[linewidth=0.04cm](3.43,1.055)(3.83,0.655)
\psline[linewidth=0.04cm](3.83,0.655)(4.23,1.055)
\psline[linewidth=0.04cm](4.23,1.055)(4.63,0.655)
\psline[linewidth=0.04cm](4.63,0.655)(5.03,1.055)
\psline[linewidth=0.04cm](5.03,1.055)(5.43,0.655)
\psline[linewidth=0.04cm](5.43,0.655)(5.83,1.055)
\psline[linewidth=0.04cm](5.83,1.055)(6.23,0.655)
\psline[linewidth=0.04cm](6.23,0.655)(6.63,1.055)
\psline[linewidth=0.04cm](6.63,1.055)(7.03,0.655)
\psline[linewidth=0.04cm](7.03,0.655)(7.43,1.055)
\psline[linewidth=0.04cm](7.43,1.055)(7.83,0.655)
\psline[linewidth=0.04cm](7.83,0.655)(8.23,1.055)
\psline[linewidth=0.04cm](8.23,1.055)(8.63,0.655)
\psline[linewidth=0.04cm](1.43,0.655)(1.83,0.255)
\psline[linewidth=0.04cm](1.83,0.255)(2.23,0.655)
\psline[linewidth=0.04cm](2.23,0.655)(2.63,0.255)
\psline[linewidth=0.04cm](2.63,0.255)(3.03,0.655)
\psline[linewidth=0.04cm](3.03,0.655)(3.43,0.255)
\psline[linewidth=0.04cm](3.43,0.255)(3.83,0.655)
\psline[linewidth=0.04cm](3.83,0.655)(4.23,0.255)
\psline[linewidth=0.04cm](4.23,0.255)(4.63,0.655)
\psline[linewidth=0.04cm](4.63,0.655)(5.03,0.255)
\psline[linewidth=0.04cm](5.03,0.255)(5.43,0.655)
\psline[linewidth=0.04cm](5.43,0.655)(5.83,0.255)
\psline[linewidth=0.04cm](5.83,0.255)(6.23,0.655)
\psline[linewidth=0.04cm](6.23,0.655)(6.63,0.255)
\psline[linewidth=0.04cm](6.63,0.255)(7.03,0.655)
\psline[linewidth=0.04cm](7.03,0.655)(7.43,0.255)
\psline[linewidth=0.04cm](7.43,0.255)(7.83,0.655)
\psline[linewidth=0.04cm](7.83,0.655)(8.23,0.255)
\psline[linewidth=0.04cm](8.23,0.255)(8.63,0.655)
\psdots[dotsize=0.06](0.43,0.655)
\psdots[dotsize=0.06](0.23,0.655)
\psdots[dotsize=0.06](0.03,0.655)
\psdots[dotsize=0.06](8.83,0.655)
\psdots[dotsize=0.06](9.03,0.655)
\psdots[dotsize=0.06](9.23,0.655)
\psline[linewidth=0.04cm](0.63,-0.745)(1.03,-0.345)
\psline[linewidth=0.04cm](0.63,-0.745)(1.03,-1.145)
\psline[linewidth=0.04cm](0.63,-0.745)(1.43,-0.745)
\psline[linewidth=0.04cm](1.03,-0.345)(1.43,-0.745)
\psline[linewidth=0.04cm](1.03,-1.145)(1.43,-0.745)
\psline[linewidth=0.04cm](1.43,-0.745)(8.63,-0.745)
\psline[linewidth=0.04cm](1.43,-0.745)(1.83,-0.345)
\psline[linewidth=0.04cm](1.83,-0.345)(2.23,-0.745)
\psline[linewidth=0.04cm](2.23,-0.745)(2.63,-0.345)
\psline[linewidth=0.04cm](2.63,-0.345)(3.03,-0.745)
\psline[linewidth=0.04cm](3.03,-0.745)(3.43,-0.345)
\psline[linewidth=0.04cm](3.43,-0.345)(3.83,-0.745)
\psline[linewidth=0.04cm](3.83,-0.745)(4.23,-0.345)
\psline[linewidth=0.04cm](4.23,-0.345)(4.63,-0.745)
\psline[linewidth=0.04cm](4.63,-0.745)(5.03,-0.345)
\psline[linewidth=0.04cm](5.03,-0.345)(5.43,-0.745)
\psline[linewidth=0.04cm](5.43,-0.745)(5.83,-0.345)
\psline[linewidth=0.04cm](5.83,-0.345)(6.23,-0.745)
\psline[linewidth=0.04cm](6.23,-0.745)(6.63,-0.345)
\psline[linewidth=0.04cm](6.63,-0.345)(7.03,-0.745)
\psline[linewidth=0.04cm](7.03,-0.745)(7.43,-0.345)
\psline[linewidth=0.04cm](7.43,-0.345)(7.83,-0.745)
\psline[linewidth=0.04cm](7.83,-0.745)(8.23,-0.345)
\psline[linewidth=0.04cm](8.23,-0.345)(8.63,-0.745)
\psline[linewidth=0.04cm](1.43,-0.745)(1.83,-1.145)
\psline[linewidth=0.04cm](1.83,-1.145)(2.23,-0.745)
\psline[linewidth=0.04cm](2.23,-0.745)(2.63,-1.145)
\psline[linewidth=0.04cm](2.63,-1.145)(3.03,-0.745)
\psline[linewidth=0.04cm](3.03,-0.745)(3.43,-1.145)
\psline[linewidth=0.04cm](3.43,-1.145)(3.83,-0.745)
\psline[linewidth=0.04cm](3.83,-0.745)(4.23,-1.145)
\psline[linewidth=0.04cm](4.23,-1.145)(4.63,-0.745)
\psline[linewidth=0.04cm](4.63,-0.745)(5.03,-1.145)
\psline[linewidth=0.04cm](5.03,-1.145)(5.43,-0.745)
\psline[linewidth=0.04cm](5.43,-0.745)(5.83,-1.145)
\psline[linewidth=0.04cm](5.83,-1.145)(6.23,-0.745)
\psline[linewidth=0.04cm](6.23,-0.745)(6.63,-1.145)
\psline[linewidth=0.04cm](6.63,-1.145)(7.03,-0.745)
\psline[linewidth=0.04cm](7.03,-0.745)(7.43,-1.145)
\psline[linewidth=0.04cm](7.43,-1.145)(7.83,-0.745)
\psline[linewidth=0.04cm](7.83,-0.745)(8.23,-1.145)
\psline[linewidth=0.04cm](8.23,-1.145)(8.63,-0.745)
\psdots[dotsize=0.06](0.43,-0.745)
\psdots[dotsize=0.06](0.23,-0.745)
\psdots[dotsize=0.06](0.03,-0.745)
\psdots[dotsize=0.06](8.83,-0.745)
\psdots[dotsize=0.06](9.03,-0.745)
\psdots[dotsize=0.06](9.23,-0.745)
\psdots[dotsize=0.12](1.83,0.255)
\psdots[dotsize=0.12](1.83,0.655)
\psdots[dotsize=0.12](1.83,1.055)
\psdots[dotsize=0.12](4.23,0.255)
\psdots[dotsize=0.12](4.23,0.655)
\psdots[dotsize=0.12](4.23,1.055)
\psdots[dotsize=0.12](6.63,0.255)
\psdots[dotsize=0.12](6.63,0.655)
\psdots[dotsize=0.12](6.63,1.055)
\psdots[dotsize=0.18](2.23,-0.745)
\psdots[dotsize=0.18](4.63,-0.745)
\psdots[dotsize=0.18](7.03,-0.745)
\psdots[dotsize=0.18](1.83,1.055)
\psdots[dotsize=0.18](1.83,0.655)
\psdots[dotsize=0.18](1.83,0.255)
\psdots[dotsize=0.18](4.23,1.055)
\psdots[dotsize=0.18](4.23,0.655)
\psdots[dotsize=0.18](4.23,0.255)
\psdots[dotsize=0.18](6.63,1.055)
\psdots[dotsize=0.18](6.63,0.655)
\psdots[dotsize=0.18](6.63,0.255)
\end{pspicture} 
}

%% file: thick.bbl
\begin{thebibliography}{10}

\bibitem{Amiot}
Claire Amiot.
\newblock On the structure of triangulated categories with finitely many
  indecomposables.
\newblock {\em Bull. Soc. Math. France}, 135(3):435--474, 2007.

\bibitem{Araya}
Tokuji Araya.
\newblock Exceptional sequences over path algebras of type {$A_n$} and
  non-crossing spanning trees.
\newblock arXiv:math.RT/0904.2831v1, 2009.

\bibitem{Asashiba}
Hideto Asashiba.
\newblock The derived equivalence classification of representation-finite
  selfinjective algebras.
\newblock {\em J. Algebra}, 214(1):182--221, 1999.

\bibitem{Assem}
Ibrahim Assem, Daniel Simson, and Andrzej Skowro{\'n}ski.
\newblock {\em Elements of the representation theory of associative algebras.
  {V}ol. 1}, volume~65 of {\em London Mathematical Society Student Texts}.
\newblock Cambridge University Press, Cambridge, 2006.
\newblock Techniques of representation theory.

\bibitem{Reiner2}
Christos~A. Athanasiadis and Victor Reiner.
\newblock Noncrossing partitions for the group {$D_n$}.
\newblock {\em SIAM J. Discrete Math.}, 18(2):397--417 (electronic), 2004.

\bibitem{Benson}
D.~J. Benson, Jon~F. Carlson, and Jeremy Rickard.
\newblock Thick subcategories of the stable module category.
\newblock {\em Fund. Math.}, 153(1):59--80, 1997.

\bibitem{Bialkowski}
Jerzy Bia{\l}kowski and Andrzej Skowro{\'n}ski.
\newblock Calabi-{Y}au stable module categories of finite type.
\newblock {\em Colloq. Math.}, 109(2):257--269, 2007.

\bibitem{Brady}
Thomas Brady.
\newblock A partial order on the symmetric group and new {$K(\pi,1)$}'s for the
  braid groups.
\newblock {\em Adv. Math.}, 161(1):20--40, 2001.

\bibitem{Bruening}
Kristian Br{\"u}ning.
\newblock Thick subcategories of the derived category of a hereditary algebra.
\newblock {\em Homology, Homotopy Appl.}, 9(2):165--176, 2007.

\bibitem{Buan}
Aslak~Bakke Buan, Robert Marsh, Markus Reineke, Idun Reiten, and Gordana
  Todorov.
\newblock Tilting theory and cluster combinatorics.
\newblock {\em Adv. Math.}, 204(2):572--618, 2006.

\bibitem{Cibils}
Claude Cibils and Eduardo~N. Marcos.
\newblock Skew category, {G}alois covering and smash product of a
  {$k$}-category.
\newblock {\em Proc. Amer. Math. Soc.}, 134(1):39--50 (electronic), 2006.

\bibitem{Dichev}
N.~Dichev.
\newblock {\em Thick subcategories for quiver representations}.
\newblock Diss., University of Paderborn, 2009.

\bibitem{Dieterich}
Ernst Dieterich.
\newblock The {A}uslander-{R}eiten quiver of an isolated singularity.
\newblock In {\em Singularities, representation of algebras, and vector bundles
  ({L}ambrecht, 1985)}, volume 1273 of {\em Lecture Notes in Math.}, pages
  244--264. Springer, Berlin, 1987.

\bibitem{Gabriel}
Peter Gabriel.
\newblock Auslander-{R}eiten sequences and representation-finite algebras.
\newblock In {\em Representation theory, {I} ({P}roc. {W}orkshop, {C}arleton
  {U}niv., {O}ttawa, {O}nt., 1979)}, volume 831 of {\em Lecture Notes in
  Math.}, pages 1--71. Springer, Berlin, 1980.

\bibitem{Happel}
Dieter Happel.
\newblock {\em Triangulated categories in the representation theory of
  finite-dimensional algebras}, volume 119 of {\em London Mathematical Society
  Lecture Note Series}.
\newblock Cambridge University Press, Cambridge, 1988.

\bibitem{Hopkins2}
Michael~J. Hopkins.
\newblock Global methods in homotopy theory.
\newblock In {\em Homotopy theory ({D}urham, 1985)}, volume 117 of {\em London
  Math. Soc. Lecture Note Ser.}, pages 73--96. Cambridge Univ. Press,
  Cambridge, 1987.

\bibitem{Hopkins}
Michael~J. Hopkins and Jeffrey~H. Smith.
\newblock Nilpotence and stable homotopy theory. {II}.
\newblock {\em Ann. of Math. (2)}, 148(1):1--49, 1998.

\bibitem{Ingalls}
Colin Ingalls and Hugh Thomas.
\newblock Noncrossing partitions and representations of quivers.
\newblock {\em Compos. Math.}, 145(6):1533--1562, 2009.

\bibitem{Keller}
Bernhard Keller.
\newblock On triangulated orbit categories.
\newblock {\em Doc. Math.}, 10:551--581, 2005.

\bibitem{Kellerdg}
Bernhard Keller.
\newblock On differential graded categories.
\newblock In {\em International {C}ongress of {M}athematicians. {V}ol. {II}},
  pages 151--190. Eur. Math. Soc., Z\"urich, 2006.

\bibitem{Kreweras}
G.~Kreweras.
\newblock Sur les partitions non crois\'ees d'un cycle.
\newblock {\em Discrete Math.}, 1(4):333--350, 1972.

\bibitem{Neeman}
Amnon Neeman.
\newblock The chromatic tower for {$D(R)$}.
\newblock {\em Topology}, 31(3):519--532, 1992.
\newblock With an appendix by Marcel B{\"o}kstedt.

\bibitem{Quillen}
Daniel Quillen.
\newblock Higher algebraic {$K$}-theory. {I}.
\newblock In {\em Algebraic {$K$}-theory, {I}: {H}igher {$K$}-theories ({P}roc.
  {C}onf., {B}attelle {M}emorial {I}nst., {S}eattle, {W}ash., 1972)}, pages
  85--147. Lecture Notes in Math., Vol. 341. Springer, Berlin, 1973.

\bibitem{Reiner}
Victor Reiner.
\newblock Non-crossing partitions for classical reflection groups.
\newblock {\em Discrete Math.}, 177(1-3):195--222, 1997.

\bibitem{Riedtmann}
C.~Riedtmann.
\newblock Algebren, {D}arstellungsk\"ocher, \"{U}berlagerungen und zur\"uck.
\newblock {\em Comment. Math. Helv.}, 55(2):199--224, 1980.

\bibitem{Riedtmann1}
Christine Riedtmann.
\newblock On stable blocks of {A}uslander-algebras.
\newblock {\em Trans. Amer. Math. Soc.}, 283(2):485--505, 1984.

\bibitem{Simion}
Rodica Simion and Daniel Ullman.
\newblock On the structure of the lattice of noncrossing partitions.
\newblock {\em Discrete Math.}, 98(3):193--206, 1991.

\bibitem{Xiao}
Jie Xiao and Bin Zhu.
\newblock Locally finite triangulated categories.
\newblock {\em J. Algebra}, 290(2):473--490, 2005.

\end{thebibliography}
